\documentclass[12pt,oneside,reqno]{amsart}
\usepackage[text={6.5in,9.5in},centering]{geometry}
\usepackage{graphicx,url,subfig}
\RequirePackage[OT1]{fontenc}
\RequirePackage{amsthm,amsmath,amssymb,amscd}
\RequirePackage[colorlinks,citecolor=blue,urlcolor=blue]{hyperref}
\usepackage{xcolor}
\usepackage{enumerate}
\usepackage{datetime}
\usepackage[normalem]{ulem} 
\usepackage{soul} 
\makeatother
\numberwithin{equation}{section}
\allowdisplaybreaks
\usepackage{dsfont}
\usepackage{mathtools}

\usepackage{amsthm}

\theoremstyle{definition}

\theoremstyle{plain}
\newtheorem{theorem}{Theorem}[section]

\newtheorem{lemma}[theorem]{Lemma}
\newtheorem{proposition}[theorem]{Proposition}

\theoremstyle{definition}

\newtheorem{assumption}[theorem]{Assumption}

\newtheorem{remark}[theorem]{Remark}

\newcommand{\E}{\mathbb{E}}
\renewcommand{\P}{\mathbb{P}}

\newcommand{\ud}{\ensuremath{\mathrm{d} }}

\newcommand{\Norm}[1]{\left\|  #1 \right\|}

\newcommand{\bbN}{\mathbb{N}}
\newcommand{\bbZ}{\mathbb{Z}}

\newcommand{\R}{\mathbb{R}}

\newcommand{\lr}[1]{\left(  #1 \right)}

\newcommand{\red}[1]{\textcolor{red}{#1}}
\newcommand{\blue}[1]{\textcolor{blue}{#1}}

\usepackage{amsrefs}

\renewcommand{\MR}[1]{~\href{https://mathscinet.ams.org/mathscinet-getitem?mr=MR#1}{MR#1}.}

\BibSpec{article}{%
    +{}  {\PrintAuthors}                {author}
    +{,} { \textit}                     {title}
    +{.} { }                            {part}
    +{:} { \textit}                     {subtitle}
    +{,} { \PrintContributions}         {contribution}
    +{.} { \PrintPartials}              {partial}
    +{,} { }                            {journal}
    +{}  { \textbf}                     {volume}
    +{}  { \PrintDatePV}                {date}
    +{,} { \issuetext}                  {number}
    +{,} { \eprintpages}                {pages}
    +{,} { }                            {status}
    +{,} { \url}                        {url}
    +{,} { \PrintDOI}                   {doi}
    +{,} { available at \eprint}        {eprint}
    +{}  { \parenthesize}               {language}
    +{}  { \PrintTranslation}           {translation}
    +{;} { \PrintReprint}               {reprint}
    +{.} { }                            {note}
    +{.} {}                             {transition}
    +{}  {\SentenceSpace \PrintReviews} {review}
}
\title[3D Stochastic Wave Equation with non-Lipschitz coefficients]{
Three-dimensional stochastic wave equation with non-Lipschitz coefficients}
\author{Jingyu Huang$^{1,\dagger}$}
\author{Wenxuan Tao$^{1,\ast}$}

\address{$^{1}$ School of Mathematics, University of Birmingham,
Birmingham,
B15 2TT,
United Kingdom}
\email{$^{\dagger}$ j.huang.4@bham.ac.uk}
\email{$^{\ast}$ wxt399@student.bham.ac.uk}
 \date{\today}
\begin{document}
\begin{abstract}
    We consider the three-dimensional stochastic wave equation (SWE) driven by a multiplicative Gaussian noise that is white in time and colored in space:  
\[
    \frac{\partial^2 u}{\partial t^2}
    = \Delta u
    + b\bigl(u\bigr)
    + \sigma\bigl(u\bigr)\,\dot{W},
\]
where the drift function \( b \) and diffusion coefficient \( \sigma \) are assumed to be locally Lipschitz and exhibit logarithmic superlinear growth at infinity. We establish the existence and uniqueness of a global mild solution on any fixed time interval \([0,T]\) under suitable assumptions on the spatial covariance function \( f \) of the noise \(\dot W(t,x)\). Our results apply, for example, to the case
\[
b(u) = u (\log_+ u)^{\theta_1}
\quad \text{and} \quad
\sigma(u) = u (\log_+ u)^{\theta_2},
\]
with parameters \(\theta_1 \in (0,2)\) and \(\theta_2 \in \bigl(0, \tfrac{\bar{\nu}+1}{2}\bigr)\), and $\log_+(z)=\log(z\vee e)$, where \(\bar{\nu}\) is determined by the assumptions on \( f \).

\end{abstract}
\maketitle
\section{Introduction}
In this paper, we study the three-dimensional stochastic wave equation driven by a multiplicative noise $\dot W$ that is white in time and correlated in space:
\begin{equation}\label{E:SWE}
    \left\{
        \begin{aligned}
            &\frac{\partial^2 u}{\partial t^2}(t,x)=\Delta u(t,x)+b\lr{u(t,x)}+\sigma\left(u(t,x)\right)\dot{W}(t,x)\,,\\
            &u(0,x)=u_0(x)\,,\quad x\in \R^3\,,\\
            &\frac{\partial u}{\partial t}(0,x)=v_0(x)\,,\quad x\in \R^3\,,
        \end{aligned}
    \right.
\end{equation}
starting from bounded initial displacement $u_0$ and initial velocity $v_0$. The drift term $b$ and diffusion coefficient $\sigma$ are assumed to be locally Lipschitz and may have faster than linear growth when $u$ is large. $\dot{W}$ is a centered Gaussian noise with covariance structure
\begin{align}\label{E:CovStructure}
    \E\left[\dot{W}(t,x)\dot{W}(s,y)\right]=\delta(t-s)f(x-y)\,,
\end{align}
where $f$ is a nonnegative and nonnegative-definite function with certain assumptions to be specified (see Assumption~\ref{A:CovAssumption}). 
Under the assumption
\begin{equation}\label{E:BasicAssump}
    \int_{|z|\le 1}\frac{f(z)}{|z|}\ud z<\infty\,,
\end{equation}
there exists a unique mild solution to \eqref{E:SWE} when both $b$ and $\sigma$ are globally Lipschitz; see, for example, \cite{dalang.quer-sardanyons:11:stochastic,peszat.zabczyk:00:nonlinear, peszat:02:cauchy}.


If $b$ and/or $\sigma$ are not globally Lipschitz in an SPDE, then blow-up can occur and a global solution may fail to exist, while a global solution exists when blow-up does not occur. For example, in the case of stochastic heat equations (SHEs)
\begin{equation}\label{E: SHE}
    \frac{\partial u}{\partial t} = \frac{1}{2} \frac{\partial^2 u}{\partial x^2} + b(u) + \sigma(u)\dot{W}\,.
  \end{equation}
with superlinear diffusion coefficients but without a drift term, Mueller~\cite{mueller:98:long-time, mueller:00:critical, mueller.sowers:93:blowup} studied the blow-up phenomena for~\eqref{E: SHE} when the diffusion coefficient $\sigma(u)$ behaves like $u^{\gamma}$. He showed that when $\gamma<\frac{3}{2}$, there is a global solution, while for $\gamma>\frac{3}{2}$, finite-time blow-up occurs with positive probability. The critical case $\gamma=\frac{3}{2}$ was recently resolved by Salins~\cite{salins:25:solutions}. 

In the last few years, there has been extensive research on SPDEs with non-Lipschitz drift and diffusion coefficients. Most of these works concern SHEs with locally Lipschitz coefficients; see \cite{fernandez-bonder.groisman:09:time-space}, \cite{dalang.khoshnevisan.ea:19:global}, 
\cite{chen.foondun.ea:25:global}, 
\cite{salins:26:global, salins:22:existence, salins:22:global}, 
\cite{shang.zhang:21:global}. For an SHE with non-locally Lipschitz coefficients, we refer to the recent paper \cite{chen.huang.ea:25:stochastic}.

Unlike stochastic heat equations, less is known for stochastic wave equations (SWEs), due to the more involved structure of the fundamental solutions of wave equations. Mueller~\cite{mueller:97:long} proved that global solutions still exist in certain non-Lipschitz regimes. In particular, he established the existence and uniqueness of global mild solutions to the stochastic wave equations on $\R$ and $\R^2$ under the assumptions
\begin{align}\label{E:MuellerAssu}
    b\equiv0,\quad |\sigma(u)|\le c\lr{|u|+1}\log\lr{|u|+2}^{\alpha},\quad\text{$\sigma$ is locally Lipschitz}\,,
\end{align}
with $\alpha\in (0,1/2)$. See also \cite{foondun.nualart:22:non-existence} for a finite-time blow-up result in the one-dimensional case. 


In this paper, we use the methodology in \cite{mueller:97:long} to study the three-dimensional case. The main difficulty is that the three-dimensional wave kernel (see~\eqref{E:Kernel} below) is a measure supported on a sphere. We will apply the approach developed in~\cite{hu.huang.ea:14:on} (which is inspired by \cite{dalang.sanz-sole:09:holder-sobolev}) to handle this difficulty. 

We now introduce the assumptions on the spatial covariance function $f$, which are slight modifications of those imposed in~\cite{hu.huang.ea:14:on}.
\begin{assumption}\label{A:CovAssumption}
There exist $\nu,\nu_1,\nu_2\in(0,\infty)$, $\gamma_1,\gamma_2,\mu_1,\mu_2\in (0,1]$, and a constant $C>0$ such that, for all $t\in [0,T]$, $w\in\R^3$ with $|w|\le 1$, and $h\in[0,1]$, the following hold:
    \begin{align}
        &\int_{|z|\le t}\frac{f(z)}{|z|}\ud z \le Ct^{\nu}\,, \label{E:ZeroOrderDifference}\\
        &\int_{|z|\le t}\frac{|f(z+w)-f(z)|}{|z|}\ud z \le C|w|^{\gamma_1}\,t^{\nu_1}\,, \label{E:FirstOrderDifference}\\
       & \int_{|z|\le t}\frac{|2f(z)-f(z+w)-f(z-w)|}{|z|}\ud z \le C|w|^{2\gamma_2}\,t^{\nu_2}\,,  \label{E:SecondOrderDifference}
    \end{align}
    \begin{align}
        &\int_0^T\int_{S^2\times S^2} s\left|f\lr{(s+h)\xi-(s+h)\eta}-f\lr{(s\xi-(s+h)\eta)}\right|A(\ud \xi )A(\ud \eta)\ud s\le Ch^{\mu_1}\,, \label{E:FirstOrderDifference_2}\\
     \notag&\int_0^T\int_{S^2\times S^2}s^2\bigg|f\lr{(s+h)\xi+(s+h)\eta}-f\lr{s\xi+(s+h)\eta}\\
    &\qquad\qquad\qquad -f\lr{(s+h)\xi+s\eta}+f(s\xi+s\eta)\bigg|A(\ud \xi)A(\ud \eta)\ud s\le Ch^{2\mu_2}\,, \label{E:SecondOrderDifference_2}
    \end{align}
    where $A(\ud \xi)$ and $A(\ud \eta)$ denote the uniform surface measure on the unit sphere $S^2\subset \R^3$.
\end{assumption}
As it turns out, these assumptions are satisfied by a variety of important covariance functions, such as the Riesz kernel and the Bessel kernel (see Section~\ref{S:Examples}). Similar to the assumptions~\eqref{E:MuellerAssu} in the one- or two-dimensional cases, we introduce the logarithmic correction condition at infinity for both the drift term and the diffusion coefficient. {The growth rates in the time variable $t$ shown in~\eqref{E:ZeroOrderDifference}--\eqref{E:SecondOrderDifference} determine the growth of moments of the solution and the spatial and temporal increments of the solution. This, in turn, is crucial to the H\"{o}lder coefficients in Proposition~\ref{P:SpatInc} and~\ref{P:TemptInc} which allows the existence of global solution under logarithmic correction condition.} 
We now state our main theorem. 
\begin{theorem}\label{T:MainTheorem}
    Assume $u_0\in C^2\lr{\R^3}$. Suppose $u_0$, $\nabla u_0$, and $v_0$ are bounded and let $u_0$, $\Delta u_0$, and $v_0$ be H\"{o}lder continuous with exponents $\alpha_1$, $\alpha_2$, and $\alpha_3$, respectively. We also assume Assumption \ref{A:CovAssumption} and that $b$ and $\sigma$ are locally Lipschitz with the following growth condition.
    \begin{equation}
        b(z)=O\lr{|z|\lr{\log |z|}^{\theta_1}}\quad\text{and}\quad \sigma(z)=O\lr{|z|\lr{\log |z|}^{\theta_2}}\quad\text{as $|z|\to \infty$}\,,
    \end{equation}
    with $0<\theta_1<  2$, $0<\theta_2<\frac{\bar{\nu}+1}{2}$ and $\bar{\nu}=\min\{\nu,\nu_1,\nu_2\}$, where $\nu,\nu_1,\nu_2$ are given by Assumption~\ref{A:CovAssumption}. Then there exists a unique global solution $u$ to~\eqref{E:SWE}.
\end{theorem}

This paper is organized as follows. In Section \ref{Sec: Preliminaries}, we present some preliminary material about the equation \eqref{E:SWE}, as well as a moment bound for the solution. In Section \ref{Sec: holder regularity}, we study the spatial and temporal H\"{o}lder regularity of the solution, which will be used in Section \ref{S:ProofMainThm}, where we prove the main theorem. In Section \ref{S:Examples}, we give some examples of covariance functions to which our main result applies. In the following, both $c$ and $C$ denote generic positive constants whose values may change from line to line. Let $(\Omega,\mathcal{F},\P)$ be a complete probability space. For each $p\ge 1$, we denote by $\|\cdot\|_p$ the $L^p(\Omega)$ norm. For any globally Lipschitz function $g$, we write $L_g$ for its Lipschitz coefficient.

\section{Preliminaries}\label{Sec: Preliminaries}

In this section we introduce some basic elements needed in the proof of the main result. To begin with, let $W=\left\{W(\phi):\phi\in C^{\infty}_c(\R_+\times \R^3)\right\}$ be a family of centered Gaussian processes defined on the probability space $(\Omega,\mathcal{F},\P)$, indexed by compactly supported smooth functions $\phi \in C^{\infty}_c(\R_+\times \R^3)\eqqcolon \mathcal{D}$. For each $\phi, \psi\in \mathcal{D}$, the covariance is given by 
\begin{align*}
    \E \lr{W(\phi)W(\psi)}=\int_0^\infty\int_{\R^3}\int_{\R^3}\phi(s,x)\psi(s,y)f(x-y)\ud x\ud y \ud s\eqqcolon \langle \phi,\psi \rangle_{\mathcal{H}}\,.
\end{align*}
We denote by $\mathcal{H}$ the Hilbert space obtained as the completion of $\mathcal{D}$ with respect to the inner product $\langle\cdot,\cdot\rangle_{\mathcal{H}}$. Then, the map $W$ defines an isometry from $ \mathcal{D}$ to $ L^2(\Omega)$, which can be extended to $\mathcal{H}$. With the convention that $\mathds{1}_{[0,y]}\coloneqq-\mathds{1}_{[y,0]}$ when $y<0$, the indicator functions of the form $\mathds{1}_{[0,t]\times [0,x_1]\times [0, x_2]\times [0, x_3]}$ for $(t,x)\coloneqq(t,x_1,x_2,x_3)\in \R_+\times \R^3$ belong to $\mathcal{H}$. With a slight abuse of notation, we write $$W(t,x)=W(\mathds{1}_{[0,t]\times [0,x_1]\times [0, x_2]\times [0, x_3]}).$$
Formally, we write $\dot W(t,x)$ for the space-time noise corresponding to $W$, defined as the distributional time-space derivative of $W(t,x)$. With this notation, $\dot W$ is a centered Gaussian random field whose covariance structure is given by~\eqref{E:CovStructure} in the distributional sense. In particular, $\dot W$ is white in time and spatially correlated with covariance function $f$. 

For $t>0$, let $G_t$ be the fundamental solution to the three-dimensional wave equation $\partial_t^2 u=\Delta u$, then
\begin{equation}\label{E:Kernel}
    G_t=\frac{1}{4\pi t}A_t\,,
\end{equation}
where $A_t$ is the uniform surface measure (with total mass $4\pi t^2$) on the sphere in $\R^3$ with radius $t$. 
As a result, we have the scaling property
\begin{equation}\label{E:Scaling}
    G_t(\ud y)=tG_1(t^{-1} \ud y).
\end{equation}
In other words, $G_t$ is $t$ times the push-forward of $G_1$ under the map $y\mapsto ty$.

A mild solution $u(t,x)$ to~\eqref{E:SWE} is defined to solve the following integral equation
\begin{equation}\label{E:MildSolu}
\begin{aligned}
    u(t,x)=&V(t,x)+\int_0^t\int_{\R^3}G_{t-s}(x-\ud y)b\lr{u(s,y)}\ud s\\
    &+\int_0^t\int_{\R^3}G_{t-s}(x-y)\sigma\lr{u(s,y)}W(\ud s,\ud y)\,, \quad a.s.
    \end{aligned}
\end{equation}
for each fixed $(t,x)$. The last integral on the right-hand side of~\eqref{E:MildSolu} is understood in the localized version of the It\^{o}--Walsh stochastic integral (see~\cite{walsh:86:introduction}), for which we only require that 
$$
\int_0^t \int_{\R^3}\int_{\R^3} G_{t-s}(x-\ud y)G_{t-s}(x-\ud y')\sigma(u(s,y))\sigma(u(s,y'))f(y-y') \ud s < \infty\,, \quad a.s.
$$
 The deterministic function
$V$ is fixed by the initial data of the stochastic wave equation~\eqref{E:SWE} throughout the paper,
\begin{equation}\label{E:DefV}
\begin{aligned}
    V(t,x)=&\frac{\ud }{\ud t}\lr{G_tu_0}(x)+\lr{G_tv_0}(x)\\
    =& \frac{1}{4 \pi t^2} \int_{\mathbb{R}^3} \left( t v_0(x-y) + u_0(x-y) + \nabla u_0(x-y) \cdot y \right) \sigma_t(\ud y)\,.
    \end{aligned}
\end{equation}
In the case where both $b$ and $\sigma$ are globally Lipschitz, existence and uniqueness of the solution are known, see e.g., \cite{dalang.quer-sardanyons:11:stochastic}. However, here we will prove it again since we need a sharp moment bound. 
\begin{proposition}\label{P:GlobalLip}
    Assume that $b$ and $\sigma$ are globally Lipschitz with Lipschitz constants $L_b$ and $L_\sigma$, respectively. Assume further that the initial condition $u_0$ is bounded with bounded derivatives and that $v_0$ is bounded. Under assumption~\eqref{E:ZeroOrderDifference}, there exists a unique random field solution $u$ to the stochastic wave equation~\eqref{E:SWE} such that, for all $p\ge 2$ and $T>0$,
    \begin{equation} \label{E:MomentBdInfinity}
        \sup_{(t,x)\in [0,T]\times \mathbb{R}^3}\mathbb{E}|u(t,x)|^p
        \le (3M)^{\frac{p}{2}} e^{cp\beta_0 T}.
    \end{equation}
    Here 
    \begin{equation}\label{E:DefM}
    M=3\lr{T\|v_0\|_{L^\infty(\R^3)}+\|u_0\|_{L^\infty(\R^3)}+T\|\nabla u_0\|_{L^\infty(\R^3)}}^2 +Cp|\sigma(0)|^2\frac{T^{\nu+1}}{\nu+1}+C|b(0)|^2T^4\,,
\end{equation}
  and 
  $\beta_0 = C L_b^{1/2} + C p^{1/(\nu+1)} L_\sigma^{2/(\nu+1)}$, and the parameter $\nu$ is defined in~\eqref{E:ZeroOrderDifference}.
\end{proposition}

Recall the following key lemma (see Lemma 6.4 in~\cite{hu.huang.ea:14:on}). 
{\begin{lemma}\label{L:KernelConv}
    For any $s\ge t$, let $G$ be the wave kernel given by~\eqref{E:Kernel}. We have the following identity
    \begin{align*}
        \lr{G_s\ast G_t}(\ud x) =\frac{1}{8\pi |x|}\mathds{1}_{[s-t,s+t]}(|x|)\ud x\,.
    \end{align*}
\end{lemma}}
\begin{proof}[Proof of Proposition~\ref{P:GlobalLip}]
    We define Picard iterations
        \begin{align}
        u^0(t,x)=&V(t,x)=\frac{\ud }{\ud t}\lr{G_tu_0}(x)+\lr{G_tv_0}(x)\,,\label{E:FirstPicard}\\
        \notag u^{n+1}(t,x)=&u^0(t,x)+\int_0^t \int_{\R^3}G_{t-s}(x-\ud y) b\lr{u^n(s,y)}\ud s\\
        &+\int_0^t\int_{\R^3}G_{t-s}(x-y)\sigma\lr{u^n(s,y)}W\lr{\ud s,\ud y}\label{E:Picard}\,.
    \end{align}
    Apply $L^p(\Omega)$ norm, the Burkholder-Davis-Gundy inequality (see, e.g., the version in Theorem~B.1 of~\cite{khoshnevisan:14:analysis}) and Minkowski inequality to~\eqref{E:Picard} to see that
    \begin{equation}\label{E:PicardPrimary}
    \begin{aligned}
            & \Norm{u^{n+1}(t,x)}_p\le |V(t,x)|+\Norm{\int_0^t\int_{\R^3}G_{t-s}(x-\ud y)b\lr{u^n(s,y)}\ud s}_p\\
    &\ +C\sqrt{p}\Bigg\|\int_0^t\int_{\R^3}\int_{\R^3}G_{t-s}(x-\ud y)G_{t-s}(x-\ud y')f(y-y')\sigma\lr{u^n(s,y)}\sigma\lr{u^n(s,y')}\ud s\Bigg\|_{\frac p2}^{\frac12}\\
   &\le |V(t,x)|+{\int_0^t\int_{\R^3}G_{t-s}(x-\ud z)\lr{|b(0)|+L_b\sup_{y\in \R^3}\Norm{u^n(s,y)}_p}}\ud s\\
    &\quad+ C\sqrt{p}\lr{\int_0^t\int_{|z|\le 2(t-s)}\frac{f(z)}{|z|}\ud z\lr{\sigma(0)^2+L_\sigma^2\sup_{y\in \R^3}\Norm{u^n
    (s,y)}_p^2\ud s}}^\frac{1}{2}\,,
    \end{aligned}
    \end{equation}
where we applied Lemma~\ref{L:KernelConv} in the second inequality. We take supremum over spatial variables, and take into account~\eqref{E:ZeroOrderDifference} to get 
\begin{equation}\label{E:PicardMoment}
\begin{aligned}
 \sup_{x\in\R^3}\Norm{u^{n+1}(t,x)}_p^2\le& 3\sup_{x\in \R^3}|V(t,x)|^2+Cp|\sigma(0)|^2\frac{t^{\nu+1}}{\nu+1}+Cb(0)^2t^4\\
 &+CL_b^2\lr{\int_0^t(t-s)\sup_{y\in \R^3}\Norm{u^n(s,y)}_p\ud s}^2\\
    &+CpL_\sigma^2\int_0^t\lr{t-s}^\nu \sup_{y\in \R^3}\Norm{u^n(s,y)}_p^2\ud s \,.
\end{aligned}
\end{equation}
For the first term on the right-hand side, we obtain from the boundedness of $v_0$, $u_0$ and $\nabla u_0$,
\begin{align*}
    \sup_{x\in\R^3}|V(t,x)|\le& \sup_{x\in\R^3}\left|\frac{1}{4 \pi t^2} \int_{\mathbb{R}^3} \left( t v_0(x-y) + u_0(x-y) + \nabla u_0(x-y) \cdot y \right) \sigma_t(\ud y)\right|\\
    \le& t\|v_0\|_{L^\infty(\R^3)}+\|u_0\|_{L^\infty(\R^3)}+t\|\nabla u_0\|_{L^\infty(\R^3)}\,.
\end{align*}

Consider the norm
  \begin{equation}\label{E:STN}
    \mathcal{N}_{\beta_0,\: p} (u) \coloneqq
    \sup_{t\in[0,T]} \sup_{x\in \R^3} e^{-\beta_0 t} \Norm{u(t,x)}_p
    \quad \text{for ${\beta_0 > 0}$ and $p \ge 2$.}
  \end{equation}
From~\eqref{E:PicardMoment}, we get
\begin{equation}\label{E:SpaceTimeNormCalculation}
    \begin{aligned}
        e^{-2\beta_0 t}\left\|u^{n+1}(t,x)\right\|_{p}^2\le& M+L_b^2 \mathcal{N}_{\beta_0,p}(u^n)^2\lr{\int_0^t e^{-\beta_0 (t-s)}(t-s) \ud s}^2\\
    &+CpL_\sigma^2\mathcal{N}_{\beta_0,p}(u^n)^2\int_0^t e^{-2\beta_0(t-s)}\lr{t-s}^\nu \ud s\\
    \le &M+\lr{\frac{CL_b^2}{\beta_0^4}+\frac{CpL_\sigma^2\Gamma(\nu+1)}{\beta_0^{\nu+1}}}\mathcal{N}_{\beta_0,\: p} (u^n)^2\,. 
    \end{aligned}
\end{equation}
Thus, by taking supremum in $x$ and $t$, we obtain
\begin{align*}
    \mathcal{N}_{\beta_0,\: p} (u^{n+1})^2\le M+\lr{\frac{CL_b^2}{\beta_0^4}+\frac{CpL_\sigma^2\Gamma(\nu+1)}{\beta_0^{\nu+1}}}\mathcal{N}_{\beta_0,\: p} (u^n)^2\,.
\end{align*}
For all $\beta_0$ such that

\begin{align}\label{E:DefBeta}
    \beta_0\ge \max\Bigl\{[4CpL_\sigma^2\Gamma(\nu+1)]^{\frac{1}{\nu+1}},\lr{4CL_b^2}^{1/4}\Bigr\}\eqqcolon\bar{\beta}_0\,,
\end{align}
we have $\frac{CL_b^2}{\beta_0^4}+\frac{CpL_\sigma^2\Gamma(\nu+1)}{\beta_0^{\nu+1}}\le 1/2$. As a result,
\begin{align}\label{E:NormIter}
    \mathcal{N}_{\beta_0,\: p} (u^{n+1})^2\le &M+\frac{1}{2}\mathcal{N}_{\beta_0,\: p} (u^n)^2\,.
\end{align}
Standard iteration of~\eqref{E:NormIter} shows that
\begin{align*}
    \mathcal{N}_{\beta_0,p}(u^{n+1})^2\le& \sum_{i=0}^{n}\lr{\frac{1}{2}}^iM+\lr{\frac{1}{2}}^{n+1}\mathcal{N}_{\beta_0,p}(u^0)^2 \le 3M\,,
\end{align*}
where $u^0(t,x)$ is the first Picard iteration~\eqref{E:FirstPicard}. Consequently, for all $n\ge 0$,
\begin{align}\label{E:UnifNormBd}
    \mathcal{N}_{\beta_0,\: p} (u^n)\le &\sqrt{3M}\,.
\end{align}
Multiply both sides with $e^{\beta_0 T}$ and take $p$-th power to obtain
\begin{align}\label{E:MomentBoundun}
    \sup_{(t,x)\in [0,T]\times \R^3}\E\left|u^n(t,x)\right|^p\le (3M)^{\frac{p}{2}}e^{cp\beta_0 T}\,,
\end{align}
where $M$ and $\beta_0$ are given by~\eqref{E:DefM} and~\eqref{E:DefBeta}, and the right-hand side is independent of $n$. Now we show that the sequence $\{u^n(t,x)\}$ converges uniformly in $L^p(\Omega)$. Following the same steps as in~\eqref{E:PicardPrimary} and the Lipschitz property of $b$ and $\sigma$, we have
\begin{align*}
    &\|u^{n+1}(t,x)-u^n(t,x)\|_p^2\\
    \le & 2\left\|\int_0^t\int_{\R^3}G_{t-s}(x-\ud y)\left[b(u^n(s,y))-b(u^{n-1}(s,y))\right]\ud s\right\|_p^2\\
    &+C{p}\Big\|\int_0^t\int_{\R^3}\int_{\R^3}G_{t-s}(x-\ud y)G_{t-s}(x-\ud y')f(y-y')\\
&\quad\times\left[\sigma(u^{n}(s,y))-\sigma(u^{n-1}(s,y))\right]\left[\sigma(u^{n}(s,y))-\sigma(u^{n-1}(s,y))\right]\ud s\Big\|_{p/2}\\
\leq&2L_b^2 \left(\int_0^t (t-s)\sup_{y \in \R^3}\|u^n(s,y) - u^{n-1}(s,y)\|_p \ud s\right)^2\\
&+Cp L_{\sigma}^2 \int_0^t \int_{|z|\leq 2(t-s)} \frac{f(z)}{|z|} \ud z\sup_{y \in \R^3} \|u^n(s,y) - u^{n-1}(s,y)\|_p^2 \ud s\,. 
\end{align*}
Using \eqref{E:ZeroOrderDifference} and applying the norm $\mathcal{N}_{\beta_0,p}$ again and from the same calculation as in~\eqref{E:SpaceTimeNormCalculation}, it follows that 
\begin{equation}
\begin{aligned}
&\mathcal{N}_{\beta_0, p}(u^{n+1}-u^n)^2\\
\leq & 2L_b^2 \left(\int_0^t (t-s)e^{-\beta_0 (t-s)} \mathcal{N}_{\beta_0, p}(u^n-u^{n-1})ds\right)^2\\
&+C p L_{\sigma}^2 \int_0^t (t-s)^{\nu} e^{-2\beta_0 (t-s)}\mathcal{N}_{\beta_0, p}(u^n-u^{n-1})^2 ds \\
\leq& \left(\frac{2L_b^2}{\beta_0^4}+ \frac{c p L_{\sigma}^2}{\beta_0^{\nu+1}} \right) \mathcal{N}_{\beta_0, p}(u^n-u^{n-1})^2\,.
\end{aligned}
\end{equation}
We see that $\{u^n\}$ is Cauchy with respect to the norm $\mathcal{N}_{\beta_0,p}$ for $\beta_0$ sufficiently large and hence with respect to the norm $\sup_{t\in [0,T]}\sup_{x\in \R^3}\Norm{u(t,x)}_p$. We denote the limit by $u(t,x)$. Passing to the limit in~\eqref{E:Picard} shows that $u$ solves~\eqref{E:SWE}. Uniqueness follows from a standard argument together with the globally Lipschitz property of $b$ and $\sigma$. Also, we obtain the moment estimate,
\begin{align}\label{E:MomentBound}
    \sup_{(t,x)\in [0,T]\times \R^3}\Norm{u(t,x )}_p^p\le (3M)^{\frac{p}{2}}e^{cp\beta_0 T}<\infty\,,
\end{align}
which completes the proof by taking $\beta_0={[4CpL_\sigma^2\Gamma(\nu+1)}]^{\frac{1}{\nu+1}}+\lr{4CL_b^2}^{1/4}$.
\end{proof}
\begin{remark}\label{R:Beta_0}
    It is clear from the proof  that~\eqref{E:UnifNormBd} still holds if $\beta_0$ is replaced by any larger number.
\end{remark}

\section{H\"{o}lder regularity}\label{Sec: holder regularity}
To prove Theorem \ref{T:MainTheorem}, we will need the estimates of the moments of the spatial and temporal increments of the solution. In this section, we will temporarily assume that $b$ and $\sigma$ are globally Lipschitz with Lipschitz coefficients $L_b$ and $L_{\sigma}$ respectively. 
\subsection{Spatial increment}\label{S:SpatInc}
\begin{proposition}\label{P:SpatInc}
    Suppose the initial values $u_0$, $v_0$ and the spatial covariance function $f$ satisfy the assumptions of Theorem~\ref{T:MainTheorem}. Let $u$ be the solution to~\eqref{E:SWE}. Then, there exist constants $c,C>0$ such that for all $T>0$, $p\ge 2$, $x,y\in \R^3$, and $|x-y|\le 1$, $t\in [0,T]$, 
    \begin{equation} \label{E:SpatialIncResult}       
     \Norm{u(t,x)-u(t,y)}_p^2\le CM{e^{c\beta t}}|x-y|^{2{\bar{\gamma}}}\,,   
    \end{equation}
    where $\beta={\sqrt{L_{b}}+p^{\frac{1}{\overline{\nu}+1}}L_{\sigma}^{\frac{2}{\overline{\nu}+1}}}$. Here $\bar\gamma = \min\{{\alpha_1},\alpha_2,\alpha_3, \gamma_1, \gamma_2\}$, $\bar{\nu}= \min\{\nu, \nu_1, \nu_2\}$, and $L_b, L_\sigma$ are the Lipschitz constants of $b$ and $\sigma$, respectively.
\end{proposition}

\begin{proof}
 The key idea of the proof is to recenter all spherical surface measures $G_{t-s}(\cdot-\ud z)$ at the origin by applying an appropriate change of variables. In the following, we apply the following notations,
\begin{equation}\label{E:NotationSigma}
    \begin{aligned}
        \Sigma_{x}(s,z)=&\sigma\lr{u(s,x-z)}\,,\\
    \Sigma_{x,y}(s,z)=&\sigma\lr{u(s,x-z)}-\sigma\lr{u(s,y-z)}\,.
    \end{aligned}
\end{equation}
Recalling the mild formulation~\eqref{E:MildSolu}, it follows that
\begin{align*}
    u(t,x)-u(t,y)=&V(t,x)-V(t,y)+\int_0^t\int_{\R^3}\bigl[G_{t-s}(x-\ud z)-G_{t-s}(y-\ud z)\bigr]b\lr{u(s,z)}\ud s\\
    &+\int_0^t\int_{\R^3}\left[G_{t-s}(x-z)-G_{t-s}(y-z)\right]\sigma(u(s,y))W(\ud s,\ud z)\,.
\end{align*}
Again applying $L^p(\Omega)$ norm, Burkholder-Davis-Gundy inequality and Minkowski inequality, we see that
\begin{equation}\label{E_:SpatialHolder}
    \begin{aligned}
        &\Norm{u(t,x)-u(t,y)}_p^2\le 3|V(t,x)-V(t,y)|^2\\
    &+3\Norm{\int_0^t\int_{\R^3}G_{t-s}(x-\ud z)b\lr{u(s,z)}\ud s-\int_0^t\int_{\R^3}G_{t-s}(y-\ud z)b\lr{u(s,z)}\ud s}_p^2\\
    &+Cp\Bigg\|\int_0^t\int_{\R^3}\int_{\R^3}\left[ G_{t-s}(x-\ud z)-G_{t-s}(y-\ud z)\right]\left[ G_{t-s}(x-\ud z')-G_{t-s}(y-\ud z')\right]\\
    &\qquad\times f(z-z')\sigma(u(s,z))\sigma(u(s,z'))\ud s\Bigg\|_{p/2}\\
    &\le 3|V(t,x)-V(t,y)|^2+3P+Cp\sum_{i=1}^4\Norm{\int_0^t\int_{\R^3}\int_{\R^3}G_{t-s}(\ud z)G_{t-s}(\ud z')h_i\ud s}_{p/2}\\
    &\eqqcolon 3|V(t,x)-V(t,y)|^2+3P +Cp\sum_{i=1}^4 Q_i\,,
    \end{aligned}
\end{equation}
where
\begin{align}\label{E:DefP}
    P=&\Norm{\int_0^t\int_{\R^3}G_{t-s}(x-\ud z)b\lr{u(s,z)}\ud s-\int_0^t\int_{\R^3}G_{t-s}(y-\ud z)b\lr{u(s,z)}\ud s}_p^2
\end{align}
and for $i=1,2,3,4$,
\begin{align*}
    Q_i=\Norm{\int_0^t\int_{\R^3}\int_{\R^3}G_{t-s}(\ud z)G_{t-s}(\ud z')h_i\ud s}_{p/2}\,. 
\end{align*}
Recalling the notations~\eqref{E:NotationSigma} and by writing $x-y\eqqcolon w$, the $h_i$ above are given by
\begin{align*}
    h_1=&f(z'-z)\Sigma_{x,y}(s,z)\Sigma_{x,y}(s,z')\,,\\
    h_2=&\left[f(z'-z+w)-f(z'-z)\right]\Sigma_{x}(s,z)\Sigma_{x,y}(s,z')\,,\\
    h_3=&\left[f(z'-z-w)-f(z'-z)\right]\Sigma_{x}(s,z')\Sigma_{x,y}(s,z)\,,\\
    h_4=&\left[2f(z'-z)-f(z'-z+w)-f(z'-z-w)\right]\Sigma_{x}(s,z)\Sigma_{x}(s,z')\,.
\end{align*}
Recalling that $V(t,x)$ is given by~\eqref{E:DefV}, we have
\begin{align*}
    3|V(t,x)-V(t,y)|^2\le& 6\left|\lr{G_tv_0}(x)-\lr{G_tv_0}(y)\right|^2+6\left|\frac{\ud }{\ud t}\lr{G_tu_0}(x)-\frac{\ud }{\ud t}\lr{G_tu_0}(y)\right|^2\,.
\end{align*}
By the H\"{o}lder regularity of $v_0$,
\begin{align*}
    \left|\lr{G_tv_0}(x)-\lr{G_tv_0}(y)\right|\le &\left|\int_{\R^3}G_{t}\lr{\ud z}|v_0(x-z)-v_0(y-z)|\right|
    \le C|w|^{\alpha_3}\,.
\end{align*}
Using the identity (see, e.g.~\cite[Equation (3.4)]{hu.huang.ea:14:on}),
\begin{align*}
    \frac{\ud }{\ud t}\lr{G_tu_0}(x)=\frac{1}{t}\int_{\R^3}u_0(x-z)G_{t}(\ud z)+\frac{1}{4\pi}\int_{|z|<1}\lr{\Delta u_0}\lr{x+tz}\ud z\,,
\end{align*}
together with the H\"{o}lder regularity of $u_0$ and $\Delta u_0$, we have
\begin{align*}
    \left|\frac{\ud }{\ud t}\lr{G_tu_0}(x)-\frac{\ud }{\ud t}\lr{G_tu_0}(y)\right|\le& \frac{1}{t}\int_{\R^3}|u_0(x-z)-u_0(y-z)|G_t(\ud z)\\
    &+\frac{1}{4\pi}\int_{|z|<1}|\lr{\Delta u_0}\lr{x+tz}-\lr{\Delta u_0}\lr{y+tz}|\\
    \le& C\lr{|w|^{\alpha_1}+|w|^{\alpha_2}}\,.
\end{align*}
Consequently, we obtain
\begin{align}\label{E:Vdiff_Spatial}
    3|V(t,x)-V(t,y)|^2
    \le C\lr{|w|^{2\alpha_1}+|w|^{2\alpha_2}+|w|^{2\alpha_3}}\,.
\end{align}
For $P$ given in~\eqref{E:DefP}, by the Lipschitz continuity of $b$ and the expression of $G_t(x)$ in~\eqref{E:Kernel}, we see that
\begin{equation}\label{E_:PEst}
    \begin{aligned}
        P=&\Norm{\int_0^t\int_{\R^3}G_{t-s}(x-\ud z)b\lr{u(s,z)}\ud s-\int_0^t\int_{\R^3}G_{t-s}(y-\ud z)b\lr{u(s,z)}\ud s}_p^2\\
    =&\Norm{\int_0^t\int_{\R^3}G_{t-s}(x-\ud z)\left[b\lr{u(s,z)}-b\lr{u(s,z+(x-y))}\right]\ud  s}_p^2\\
    \le& L_{b}^2\lr{\int_0^t\int_{\R^3}G_{t-s}(x-\ud z)\sup_{|\xi-\xi'|=|w|}\Norm{u(s,\xi)-u(s,\xi')}_p\ud s}^2\\
    =&CL_{b}^2\left(\int_0^t\lr{t-s}\sup_{|\xi-\xi'|=|w|}\Norm{u(s,\xi)-u(s,\xi')}_p\ud s\right)^2\,.
    \end{aligned}
\end{equation}
For $\sum_{i=1}^4Q_i$, by the Lipschitz continuity of $\sigma$ and the assumption~\eqref{E:ZeroOrderDifference}, we obtain
\begin{equation}\label{E_:Q_1}
    \begin{aligned}
        Q_1=&\Bigg\|\int_0^t\int_{\R^3}\int_{\R^3}G_{t-s}(\ud z)G_{t-s}(\ud z')f(z-z')\\
    &\times \left[\sigma(u(s,x-z))-\sigma(u(s,y-z))\right]\left[\sigma(u(s,x-z'))-\sigma(u(s,y-z'))\right]\ud s\Bigg\|_{p/2}\\
    \le&L_{\sigma}^2\int_0^t\lr{\int_{|z|\le 2(t-s)}\frac{f(z)}{8\pi |z|}\ud z}\sup_{|\xi-\xi'|=|w|}\Norm{u(s,\xi)-u(s,\xi')}_p^2\ud s\\
    \le& CL_{\sigma}^2\int_0^t (t-s)^{\nu}\sup_{|\xi-\xi'|=|w|}\Norm{u(s,\xi)-u(s,\xi')}_p^2\ud s\,.
    \end{aligned}
\end{equation}
We separate $Q_2$ into two terms and apply again the Lipschitz continuity of $\sigma$ to get
\begin{align*}
    Q_2=&\Bigg\|\int_0^t\int_{\R^3}\int_{\R^3}G_{t-s}(\ud z) G_{t-s}(\ud z')\left[f(z-z'+w)-f(z-z')\right]\\
    &\times \sigma(u(s,x-z'))\lr{\sigma(u(s,x-z))-\sigma(u(s,y-z))}\ud s\Bigg\|_{p/2}\\
    \le & \frac{1}{2}\Norm{\int_0^t\int_{\R^3}\int_{\R^3}G_{t-s}(\ud z)G_{t-s}(\ud z')|w|^{\gamma_1}|f(z-z'+w)-f(z-z')||\sigma(u(s,x-z'))|^2\ud s}_{p/2}\\
    &+\frac{1}{2}L_{\sigma}^2\Bigg\|\int_0^t\int_{\R^3}\int_{\R^3}G_{t-s}(\ud z)G_{t-s}(\ud z')\frac{|f(z-z'+w)-f(z-z')|}{|w|^{\gamma_1}}\\
    &\quad\quad\quad\quad\times |u(s,x-z)-u(s,y-z)|^2\ud s\Bigg\|_{p/2}\\
    \eqqcolon& Q_{21}+Q_{22}\,,
\end{align*}
For both $Q_{21}$ and $Q_{22}$, we apply the convolution property~\eqref{E:FirstOrderDifference} and the Lipschitz continuity of $\sigma$ to get
\begin{align*}
    Q_{21}\le &C|w|^{\gamma_1} L_{\sigma}^2\int_0^t\int_{|z|\le 2\lr{t-s}}\frac{|f(z+w)-f(z)|}{|z|}\ud z\lr{1+\sup_{x\in\R^3}\Norm{u(s,x)}_p^2}\ud s\\
    \le &C|w|^{2\gamma_1}L_{\sigma}^2\int_0^t(t-s)^{\nu_1}\sup_{x\in\R^3}\Norm{u(s,x)}_p^2\ud s\,,
\end{align*}
and
\begin{align*}
    Q_{22}\le &CL_{\sigma}^2\int_0^t\int_{|z|\le 2(t-s)}\frac{|f(z+w)-f(z)|}{|w|^{\gamma_1}|z|}\ud z\sup_{|\xi-\xi'|=|w|}\Norm{u(s,\xi)-u(s,\xi')}_p^2\ud s\\
    \le &CL_{\sigma}^2\int_0^t \lr{t-s}^{\nu _1}\sup_{|\xi-\xi'|=|w|}\Norm{u(s,\xi )-u(s,\xi ')}_p^2\ud s\,.
\end{align*}
In $Q_{21}$ we have bounded $\frac{|\sigma(0)|}{L_{\sigma}}$ by a constant since finally in this paper $L_{\sigma}$ will be large, see \eqref{E:CutoffLip}. 
Consequently, we have
\begin{equation}\label{E_:Q_2}
    \begin{aligned}
         Q_2\le& C|w|^{2\gamma_1}L_{\sigma}^2\int_0^t(t-s)^{\nu_1}\lr{1+\sup_{x\in\R^3}\Norm{u(s,x)}_p^2}\ud s\\
    &+CL_{\sigma}^2\int_0^t \lr{t-s}^{\nu _1}\sup_{|\xi-\xi'|=|w|}\Norm{u(s,\xi )-u(s,\xi ')}_p^2\ud s\,.
    \end{aligned}
\end{equation}
A similar estimate also works for $Q_3$ and we get
\begin{equation}\label{E_:Q_3}
    \begin{aligned}
         Q_3\le& C|w|^{2\gamma_1}L_{\sigma}^2\int_0^t(t-s)^{\nu_1}\lr{1+\sup_{x\in\R^3}\Norm{u(s,x)}_p^2}\ud s\\
    &+CL_{\sigma}^2\int_0^t \lr{t-s}^{\nu _1}\sup_{|\xi-\xi'|=|w|}\Norm{u(s,\xi )-u(s,\xi ')}_p^2\ud s\,.
    \end{aligned}
\end{equation}
For $Q_4$, we apply the convolution property~\eqref{L:KernelConv} and the assumption~\eqref{E:SecondOrderDifference} to find
\begin{equation}\label{E_:Q_4}
    \begin{aligned}
        Q_4=&\Bigg\|\int_0^t\int_{\R^3}\int_{\R^3}G_{t-s}(\ud z)G_{t-s}(\ud z')\sigma(u(s,z))\sigma(u(s,z'))\\
    &\times \left[2f(z-z')-f(z-z'+w)-f(z-z'-w)\right]\ud s\Bigg\|_{p/2}\\
    \le &CL_{\sigma}^2 \int_0^t\int_{|z|\le 2\lr{t-s}}\frac{|2f(z)-f(z+w)-f(z-w)|}{|z|}\ud z\lr{1+\sup_{x\in\R^3}\Norm{u(s,x)}_p^2}\ud s\\
    \le & CL_{\sigma}^2|w|^{2\gamma_2}\int_0^t\lr{t-s}^{\nu_2}\lr{1+\sup_{x\in\R^3}\Norm{u(s,x)}_p^2} \ud s \,.
    \end{aligned}
\end{equation}
Gathering the bounds in~\eqref{E_:PEst}--\eqref{E_:Q_4} and substituting them into~\eqref{E_:SpatialHolder}, we obtain
\begin{align*}
    \Norm{u(t,x)-u(t,y)}_p^2\le &3|V(t,x)-V(t,y)|^2\\
    &+3L_{b}^2\left(\int_0^t\lr{t-s}\sup_{|\xi-\xi'|=|w|}\Norm{u(s,\xi)-u(s,\xi')}_p\ud s\right)^2\\
    &+cpL_{\sigma}^2\int_0^t \lr{(t-s)^\nu+(t-s)^{\nu_1}}\sup_{|\xi-\xi'|=|w|}\Norm{u(s,\xi)-u(s,\xi')}_p^2\ud s\\
    &+cpL_{\sigma}^2|w|^{2\gamma_1}\int_0^t(t-s)^{\nu_1}\lr{1+\sup_{x\in\R^3}\Norm{u(s,x)}_p^2}\ud s\\
    &+cpL_{\sigma}^2|w|^{2\gamma_2}\int_0^t\lr{t-s}^{\nu_2}\lr{1+\sup_{x\in\R^3}\Norm{u(s,x)}_p^2} \ud s\,.
\end{align*}
For the last two terms on the right-hand side, we recall the moment estimate~\eqref{E:MomentBdInfinity}. For $i=1,2$, we obtain
\begin{align*}
    pL_{\sigma}^2\int_0^t(t-s)^{\nu_i}\lr{1+\sup_{x\in\R^3}\Norm{u(s,x)}_p^2}\ud s\le &CpL_{\sigma}^2\int_0^t(t-s)^{\nu_i}Me^{c\beta s}\ud s\\
    \le &CpL_{\sigma}^2Me^{c\beta t}\int_0^\infty s^{\nu_i}e^{-c\beta s}\ud s\\
    \le &CpL_{\sigma}^2 Me^{c\beta t}\frac{\Gamma(\nu_i+1)}{\beta^{\nu_i+1}}\\
    \le& CMe^{c\beta t}
\end{align*}
from our choice of $\beta$ (see Remark~\ref{R:Beta_0}). As a result, combining all the estimates and denoting by 
$$g(s,w)^2=\sup_{|\xi-\xi'|=|w|}\Norm{u(s,\xi)-u(s,\xi')}_p^2,
$$
we have that 
\begin{equation}
\begin{aligned}
\|u(t,x) - u(t,y)\|_p^2 \leq& CM e^{c\beta t} |w|^{2\bar{\gamma}} + 3 L_b^2 \left(\int_0^t (t-s)g(s,w)ds\right)^2\\
& + Cp L_{\sigma}^2 \int_0^t [(t-s)^{\nu} + (t-s)^{\nu_1}] g^2(s,w)ds\,,
\end{aligned}
\end{equation}
thus, by taking the sup of $s$ on $[0,t]$ both sides we get
\begin{equation}
\begin{aligned}
\sup_{s\leq t}e^{-2\beta s} g^2(s,w) \leq & C M e^{c\beta t} |w|^{2\bar{\gamma}} + \frac{3L_b^2}{\beta^4} \sup_{s\leq t} e^{-2\beta s} g^2(s,w)\\
& + C p L_{\sigma}^2 \left(\frac{1}{\beta^{\nu+1}} + \frac{1}{\beta^{\nu_1+1}}\right) \sup_{s \leq t} e^{-2\beta s} g^2(s,w)\,.
\end{aligned}
\end{equation}
If we choose 
$$
\beta = C (\sqrt{L_b} + p^{\frac{1}{\bar{\nu}+1}}L_{\sigma}^{\frac{2}{\bar{\nu}+1}})
$$
for some $C$ large but fixed and independent of $w, t$, we obtain that 
$$
\sup_{s\leq t} e^{-2\beta s} g^2(s,w)\leq M |w|^{2\bar{\gamma}}\,,
$$
which leads to the proof. 
\end{proof}

\bigskip

\subsection{Temporal increment}
We can get the temporal increment estimate based on the spatial increment. The key ingredient of the proof is the scaling property~\eqref{E:Scaling}.
\begin{proposition}\label{P:TemptInc}
    Suppose the initial values $u_0$, $v_0$ and the spatial covariance function $f$ satisfy the assumptions of Theorem~\ref{T:MainTheorem}. Let $u$ be the solution to~\eqref{E:SWE}. Then, there exist constants $c,C>0$ such that for all $p\ge 2$, $T>0$, $0\le t\le  t'\le T$, $t'-t<1$ and $x\in\R^3$,
    \begin{align}\label{E:TempIncResult}
    \Norm{u (t',x)-u (t,x)}_p^2\le& C\left[1+\lr{L_{b}^2+pL_{\sigma }^2}Me^{c\beta t'}\right]|t'-t|^{2\mu}\,,
\end{align}
where $0<\mu<\overline{\mu}=\frac{1}{2}\min \{2\mu_2,\nu+1,2\overline{\gamma},\overline{\gamma}+\mu_1\}$ with $\bar\gamma = \min\{{\alpha}_1, \alpha_2, \alpha_3, \gamma_1, \gamma_2\}$, $\beta={\sqrt{L_{b }}+p^{\frac{1}{\overline{\nu}+1}}L_{\sigma }^{\frac{2}{\overline{\nu}+1}}}$ with $\bar{\nu} = \min\{\nu, \nu_1, \nu_2\}$,   and $L_b, L_\sigma$ are the Lipschitz constants of $b$ and $\sigma$ respectively, $M$ is defined in \eqref{E:DefM}. 
\end{proposition}

\begin{remark}
    As can be seen from the proof of \eqref{E: R_4} below we can take $\mu=\overline{\mu}$ in the above proposition given that $\overline{\mu}<1$. That is, at least one of the parameters $\alpha_1,\alpha_2,\alpha_3,\nu,\gamma_1,\gamma_2,\mu_1,\mu_2$ does not equal to $1$.
\end{remark}
\begin{proof}

We begin with the same treatment as in Section~\ref{S:SpatInc}. We have
\begin{align}\label{E_:TempAll}
    \notag \Norm{u (t',x)-u (t,x)}_p^2\le& 3V(t',t,x)^2+3\Bigg\|\int_0^t\int_{\R^3}G_{t-s}(x-\ud z)b \lr{u (s,z)}\ud s\\
    \notag &\qquad\qquad\qquad\qquad-\int_0^{t'}\int_{\R^3}G_{t'-s}\lr{x-\ud z}b \lr{u (s,z)}\ud s\Bigg\|_p^2\\
    \notag&+3\Bigg\|\int_0^t\int_{\R^3}G_{t-s}(x-z)\sigma \lr{u (s,z)}W(\ud s,\ud z)\\
    \notag &\qquad\quad-\int_0^{t'}\int_{\R^3}G_{t'-s}\lr{x-z}\sigma \lr{u (s,z)}W(\ud s,\ud z)\Bigg\|_p^2\\
    \leq& 3V(t',t,x)^2+6S_{1}+6S_{2}+6S_3+6S_4\,,
\end{align}
where
\begin{align*}
    V(t',t,x)^2=&\left|\lr{G_tv_0}(x)-\lr{G_{t'}v_0}(x)+\frac{\ud }{\ud t}\lr{G_tu_0}(x)-\frac{\ud }{\ud t}\lr{G_{t'}u_0}(x)\right|^2\,,\,\\
    S_1=&\Norm{\int_t^{t'}\int_{\R^3}G_{t'-s}(x-\ud z)b \lr{u (s,z)}\ud s}_p^2\,,\\
    S_2=&\Norm{\int_0^t\int_{\R^3}\left[G_{t'-s}(x-\ud z)-G_{t-s}(x-\ud z)\right]b \lr{u (s,z)}\ud s}_p^2\,,\\
    S_3=&\Norm{\int_t^{t'}\int_{\R^3}G_{t'-s}(x-z)\sigma \lr{u (s,z)}W(\ud s,\ud z)}_p^2\,,\\
    S_4=&\Norm{\int_0^t\int_{\R^3}\left[G_{t'-s}(x-z)-G_{t-s}(x-z)\right]\sigma \lr{u (s,z)}W(\ud s,\ud z)}_p^2\,.
\end{align*}
We first apply \cite[Lemma~4.9]{dalang.sanz-sole:09:holder-sobolev} to obtain
\begin{align*}
    V(t',t,x)^2\le C\lr{(t'-t)^{2\alpha_2}+(t'-t)^{2\alpha_3}}\,.
\end{align*}
For $S_1$, we apply the linear growth property of $b $ and the moment estimate~\eqref{E:MomentBdInfinity} to find
\begin{align*}
    S_1\le& \lr{\int_t^{t'}\int_{\R^3}G_{t'-s}(x-\ud z)\lr{\left|b (0)\right|+L_{b }\Norm{u (s,z)}_p }\ud s}^2\\
    \le & b (0)^2(t'-t)^4+2L_{b }^2\lr{\int_t^{t'}(t'-s)\sup_{z\in\R^3}\Norm{u (s,z)}_p\ud s}^2\\
    \le &b (0)^2(t'-t)^4+CL_{b }^2(t'-t)^4 e^{c\beta t'}\,.
\end{align*}
According to the scaling property~\eqref{E:Scaling}, we see that
\begin{align*}
    S_2=&\Bigg\|\int_0^t\int_{\R^3}G_{t'-s}(\ud z)b \lr{u (s,x-z)}\ud s
    -\int_0^t\int_{\R^3}G_{t-s}(\ud z)b \lr{u (s,x-z)}\ud s\Bigg\|_p^2\\
    = &\Bigg\|\int_0^t\int_{\R^3}\lr{t'-s}G_{1}\lr{\frac{\ud z}{t'-s}}b \lr{u (s,x-z)}\ud s\\
    &-\int_0^t\int_{\R^3}\lr{t-s}G_{1}\lr{\frac{\ud z}{t-s}}b \lr{u (s,x-z)}\ud s\Bigg\|_p^2\\
    =&\Bigg\|\int_0^t\int_{\R^3}\lr{t'-s}G_{1}({\ud z})b \lr{u (s,x-(t'-s)z)}\ud s\\
    &-\int_0^t\int_{\R^3}\lr{t-s}G_{1}(\ud z)b \lr{u (s,x-(t-s)z)}\ud s\Bigg\|_p^2\\
    \le& 2\Norm{\int_0^t\int_{\R^3}(t'-t)G_1(\ud z)b \lr{u (s,x-(t'-s)z)}\ud s}_p^2\\
    &+2\Norm{\int_0^t\int_{\R^3}(t-s)G_1(\ud z)\lr{b \lr{u (s,x-(t'-s)z)}-b \lr{u (s,x-(t-s)z)}}\ud s}_p^2\\
    =&2S_{2,1}+2S_{2,2}\,.
\end{align*}
Using the moment estimate~\eqref{E:MomentBound}, it holds that
\begin{align*}
    S_{2,1}\le& (t'-t)^2\lr{\int_0^t \int_{\R^3}G_1(\ud z)\lr{|b (0)|+L_{b }\Norm{u (s,x-(t'-s)z)}_p}\ud s}^2\\
    \le &2T^2{b (0)}^2(t'-t)^2+CT^2(t'-t)^2L_{b }^2M{e^{c\beta t}}\,.
\end{align*}
For the second term $S_{2,2}$, we apply Proposition~\ref{P:SpatInc} to obtain
\begin{align*}
    S_{2,2}\le &L_{b }^2\lr{\int_0^t (t-s)\sup_{|\xi-\xi'|=|t'-t|}\Norm{u (s,\xi)-u (s,\xi')}_p\ud s}^2\\
    \le &CT^4L_{b }^2(t'-t)^{2\overline{\gamma}}Me^{c\beta t}\,.
\end{align*}
Thus, we combine the estimates for $S_{2,1}$ and $S_{2,2}$ to get
\begin{align*}
    S_2\le CT^2(b(0)^2+L_b^2Me^{c\beta t})(t'-t)^2+CT^4ML_b^2e^{c\beta t}(t'-t)^{2\bar{\gamma}}\,.
\end{align*}
For $S_3$, by Burkholder-Davis-Gundy inequality, Lemma \ref{L:KernelConv} and \eqref{E:ZeroOrderDifference}, proceeding as before, we obtain that 
\begin{align*}
    S_3\le& CpL_{\sigma }^2\int_t^{t'}\int_{|z|\le 2(t'-s)}\frac{f(z)}{|z|}\ud z \sup_{z\in\R^3}\Norm{u (s,z)}_p^2\ud s\\
    \le & CpL_{\sigma }^2\int_t^{t'}\lr{t'-s}^\nu \sup_{z\in\R^3}\Norm{u (s,z)}_p^2\ud s\\
    \le & CpL_{\sigma}^2\frac{(t'-t)^{\nu+1}}{\nu+1}Me^{c\beta t'}\,.
\end{align*}
For $S_4$, let $t'-t = h$, another application of Burkholder-Davis-Gundy inequality shows that
\begin{align*}
    S_4\le&Cp\bigg\|\int_0^t\int_{\R^3}\int_{\R^3}[G_{t'-s}(x-\ud z)-G_{t-s}(x-\ud z)][G_{t'-s}(x-\ud z')-G_{t-s}(x-\ud z')]\\
    &\quad\quad\quad\times\sigma(u(s,z))\sigma(u(s,z')) f(z-z')\ud s\bigg\|_{\frac{p}{2}}\\
    \le &Cp\bigg\|\int_0^t\ud s\int_{\R^3}\int_{\R^3} G_{s+h}(\ud z) G_{s+h}(\ud z')\sigma(u(t-s,x-z))\sigma(u(t-s,x-z'))f(z-z')\\
    &\quad\quad\quad\quad\quad- G_{s+h}(\ud z) G_{s}(\ud z')\sigma(u(t-s,x-z))\sigma(u(t-s,x-z'))f(z-z')\\
    &\quad\quad\quad\quad\quad- G_{s}(\ud z) G_{s+h}(\ud z')\sigma(u(t-s,x-z))\sigma(u(t-s,x-z'))f(z-z')\\
    &\quad\quad\quad\quad\quad+ G_{s}(\ud z) G_{s}(\ud z')\sigma(u(t-s,x-z))\sigma(u(t-s,x-z'))f(z-z')\bigg\|_{\frac{p}{2}}
\end{align*}
By noticing that both $z$ and $z'$ are supported on a sphere, we make use of the change of coordinates $\xi=\frac{z}{|z|}$ and $\eta=\frac{z'}{|z'|}$. Let $A(\ud \xi)$ and $A(\ud \eta)$ be the uniform measure in $S^2$, so we obtain the following identities.
\begin{align*}
    G_s(\ud z)=&\frac{s}{4\pi }A(\ud \xi)\,,\\
    G_s(\ud z')=&\frac{s}{4\pi }A(\ud \eta)\,.
\end{align*}
With this change of variable we obtain
\begin{align*}
    S_4\le &Cp\bigg\|\int_0^t \ud s\int_{\R^3}\int_{\R^3} (s+h)^2\sigma(u(t-s,x-(s+h)\xi))\\
    &\quad\quad\quad\quad\quad\quad\quad\times\sigma(u(t-s,x-(s+h)\eta))f((s+h)\xi-(s+h)\eta)\\
    &\quad\quad\quad\quad\quad- s(s+h)\sigma(u(t-s,x-(s+h)\xi))\sigma(u(t-s,x-s\eta))f((s+h)\xi-s\eta)\\
    &\quad\quad\quad\quad\quad- s(s+h)\sigma(u(t-s,x-s\xi))\sigma(u(t-s,x-(s+h)\eta))f(s\xi-(s+h)\eta)\\
    &\quad\quad\quad\quad\quad+ s^2\sigma(u(t-s,x-s\xi))\sigma(u(t-s,x-s\eta))f(s\xi-s\eta)A(\ud \xi)A(\ud \eta)\bigg\|_{\frac{p}{2}}\,.
\end{align*}
In order to apply the Lipschitz condition of $\sigma$ and Assumption~\ref{A:CovAssumption}, we insert intermediate terms to separate the integral in $S_4$ into four parts.
\begin{align*}
    S_4\le \sum_{i=1}^4 R_i\,,
\end{align*}
where
\begin{align*}
    R_1=&Cp\bigg\|\int_0^t\int_{S^2\times S^2}\lr{s+h}^2f\lr{(s+h)\xi-(s+h)\eta}\\
    &\quad\quad\times\left|\sigma \lr{u (t-s,x-(s+h)\xi)}-\sigma \lr{u (t-s,x-s\xi)}\right|\\
    &\quad\quad\times \left|\sigma \lr{u (t-s,x-(s+h)\eta)}-\sigma \lr{u (t-s,x-s\eta)}\right|A(\ud \xi)A(\ud \eta)\ud s\bigg\|_\frac{p}{2}\,,\\
    R_2=& Cp \bigg\|\int_0^t\int_{S^2\times S^2}\left|(s+h)^2f\lr{(s+h)\xi-(s+h)\eta}-s(s+h)f\lr{s\xi-(s+h)\eta}\right|\\
    &\quad\quad\times \left|\sigma \lr{u (t-s,x-(s+h)\eta)}-\sigma \lr{u (t-s,x-s\eta)}\right|\\
    &\quad\quad\times|\sigma \lr{u \lr{t-s,x-s\xi}}|A(\ud \xi )A(\ud \eta)\ud s\bigg\|_{\frac{p}{2}}\,,\\
    R_3=&Cp\bigg\|\int_0^t\int_{S^2\times S^2}\left|\lr{s+h}^2f\lr{\lr{s+h}\xi-\lr{s+h}\eta}-s(s+h)f\lr{\lr{s+h}\xi-s\eta}\right|\\
    &\quad\quad\times \left|\sigma \lr{u \lr{t-s,x-\lr{s+h}\xi}}-\sigma \lr{u \lr{t-s,x-s\xi}}\right|\\
    &\quad\quad\times |\sigma\lr{u \lr{t-s,x-s\eta}}|A(\ud \xi)A(\ud \eta)\ud s\bigg\|_{\frac{p}{2}}\,,\\
    R_4 =& Cp\bigg\|\int_0^t \int_{S^2 \times S^2} \Big| (s+h)^2 f((s+h)\xi - (s+h)\eta) - s(s+h)f(s\xi - (s+h)\eta)  \\
    &\qquad\qquad\qquad\quad- s(s+h) f((s+h)\xi - s\eta) + s^2 f(s\xi - s\eta) \Big| \\
    &\quad\quad\quad\quad\quad\times |\sigma\lr{u \lr{t-s,x-s\xi}}\sigma \lr{u \lr{t-s,x-s\eta}}| A(\ud\xi) A(\ud\eta) ds\bigg\|_{\frac{p}{2}}\,.
\end{align*}
We can apply the Lipschitz property of $\sigma $ as well as the convolution property Lemma~\ref{L:KernelConv} to obtain
\begin{align*}
    R_1\le& Cp\Bigg\|\int_0^t\int_{\R^3\times \R^3}f\lr{z-z'}G_{s+h}(\ud z)G_{s+h}(\ud z')\\
    &\quad\quad\quad\times\left[\sigma \lr{u (t-s,x-z)}-\sigma \lr{u (t-s,x-\frac{s}{s+h}z)}\right]\\
    &\quad\quad\quad\times \left[\sigma \lr{u (t-s,x-z')}-\sigma \lr{u (t-s,x-\frac{s}{s+h}z')}\right]\ud s\Bigg\|_\frac{p}{2}\\
    \le & CpL_{\sigma }^2h^{2\overline{\gamma}} \int_0^t e^{c\beta(t-s)}\int_{|z|\le 2(s+h)}\frac{f(z)}{|z|}\ud z\ud s\le {CpL_{\sigma }^2h^{2\overline{\gamma}}  }e^{c\beta t}\,.
\end{align*}
Split $R_2$ into two terms to get
\begin{align*}
    R_2\le& Cp\Bigg\|\int_0^t\int_{S^2\times S^2}s(s+h)\left|f\lr{(s+h)\xi-(s+h)\eta}-f\lr{s\xi-(s+h)\eta}\right|\\
    &\quad\quad\times \left|\sigma \lr{u (t-s,x-(s+h)\eta)}-\sigma \lr{u (t-s,x-s\eta)}\right|\\
    &\quad\quad\times\left|\sigma \lr{u \lr{t-s,x-s\xi}}\right|A(\ud \xi )A(\ud \eta)\ud s\Bigg\|_{\frac{p}{2}}\\
    &+Cph\Bigg\|\int_0^t\int_{S^2\times S^2}(s+h)f\lr{(s+h)\xi-(s+h)\eta}\left|\sigma \lr{u (t-s,x-s\xi)}\right|\\
    &\quad\quad\quad\quad\times \left|\sigma \lr{u (t-s,x-(s+h)\eta)}-\sigma \lr{u (t-s,x-s\eta)}\right|A(\ud \xi)A(\ud \eta)\ud s\Bigg\|_{\frac{p}{2}}\\
    \eqqcolon& R_{2,1}+R_{2,2}\,.
\end{align*}
For $R_{2,1}$, we can apply Proposition~\ref{P:SpatInc} with parameter $\gamma=\bar{\gamma}$ together with~\eqref{E:FirstOrderDifference_2} to get
\begin{align*}
    R_{2,1}\le&  CpL_{\sigma }^2e^{c\beta t}h^{\overline{\gamma}}\\
    &\times \int_0^t\int_{S^2\times S^2} s\left|f\lr{(s+h)\xi-(s+h)\eta}-f\lr{(s\xi-(s+h)\eta)}\right|A(\ud \xi )A(\ud \eta)\ud s\\
    \le&  CpL_{\sigma }^2e^{c\beta t}h^{\overline{\gamma}+\mu_1}\,.
\end{align*}
Again by Proposition~\ref{P:SpatInc} and~\eqref{E:ZeroOrderDifference}, we have
\begin{align*}
    R_{2,2}\le &  CpL_{\sigma }^2e^{c\beta t}h^{\overline{\gamma}+1}\\
    &\times\int_0^t \frac{1}{s + h} \iint_{\mathbb{R}^3 \times \mathbb{R}^3} f(y - z) G_{s+h}(\ud y) G_{s+h}(\ud z) \, \ud s   \\
    =&CpL_{\sigma }^2e^{c\beta t}h^{\overline{\gamma}+1} \int_0^t \frac{1}{s + h} \int_{|z| \leq 2(s + h)} \frac{f(z)}{|z|} \, \ud z \, \ud s   \\
    \le& CpL_{\sigma }^2e^{c\beta t}h^{\overline{\gamma}+1}\int_0^T (s+h)^{\nu-1}\ud s\le CpL_{\sigma }^2e^{c\beta t}h^{\overline{\gamma}+1}\,.
\end{align*}
Thus we obtain that 
\begin{equation}\label{E: R_2}
R_2 \leq C p L_{\sigma}^2 e^{c\beta t} (h^{\bar{\gamma}+\mu_1} + h^{\bar{\gamma}+1})\,.
\end{equation}
The estimate for $R_3$ is the same as $R_2$. For $R_4$, we first use the moment estimate~\eqref{E:MomentBound} to get
\begin{align*}
    R_4\le& Cp\lr{|\sigma (0)|^2+L_{\sigma }^2Me^{c\beta t}}\\
    &\times\int_0^t\int_{S^2\times S^2}\Bigg[\lr{s+h}^2f\lr{(s+h)\xi-(s+h)\eta}-s(s+h)f\lr{s\xi-(s+h)\eta}\\
    &-s(s+h)f\lr{(s+h)\xi-s\eta}+s^2f\lr{s\xi-s\eta}\Bigg]A(\ud \xi)A(\ud \eta)\ud s\,.
\end{align*}
By writing $(s+h)^2=s^2+sh+sh+h^2$, we split the integral into four terms,
\begin{align*}
    &\int_0^t\int_{S^2\times S^2}\Bigg(\lr{s+h}^2f\lr{(s+h)\xi-(s+h)\eta}-s(s+h)f\lr{s\xi-(s+h)\eta}\\
    &\quad\quad\quad-s(s+h)f\lr{(s+h)\xi-s\eta}+s^2f\lr{s\xi-s\eta}A(\ud \xi)A(\ud \eta)\ud s\\
    &\le \int_0^t\int_{S^2\times S^2}s^2\Big|f\lr{(s+h)\xi+(s+h)\eta}-f\lr{s\xi+(s+h)\eta}\\
    &\qquad\quad-f\lr{(s+h)\xi+s\eta}+f(s\xi+s\eta)\Big|A(\ud \xi)A(\ud \eta)\ud s\\
    &\quad+\int_0^t\int_{S^2\times S^2} sh\left|f\lr{(s+h)\xi+(s+h)\eta}-f\lr{s\xi+(s+h)\eta}\right|A(\ud \xi )A(\ud \eta)\ud s\\
    &\quad +\int_0^t \int_{S^2\times S^2}sh\left|f\lr{(s+h)\xi+(s+h)\eta}-f\lr{(s+h)\xi+s\eta}\right|A(\ud \xi )A(\ud \eta)\ud s\\
    &\quad +\int_0^t\int_{S^2\times S^2}h^2f\lr{(s+h)\xi-(s+h)\eta}A(\ud \xi)A(\ud \eta)\ud s\\
    &\eqqcolon \sum_{i=1}^4R_{4,i}\,.
\end{align*}
By assumptions~\eqref{E:FirstOrderDifference_2} and~\eqref{E:SecondOrderDifference_2}, we have
\begin{align*}
    R_{4,1}\le Ch^{2\mu_2},\quad R_{4,2}\le Ch^{\mu_1+1},\quad\text{and}\quad R_{4,3}\le Ch^{\mu_1+1}\,.
\end{align*}
For $R_{4,4}$, we change variables from $(s+h)\xi=z$ and $(s+h)\eta=z'$ back to $z$ and $z'$, and apply the convolution property~\eqref{L:KernelConv} to obtain
\begin{align*}
    R_{4,4}=&\int_0^t\int_{S^2\times S^2}h^2f\lr{(s+h)\xi-(s+h)\eta}A(\ud \xi)A(\ud \eta)\ud s\\
    =& Ch^2 \int_0^t\int_{|z|\le 2(s+h)}\frac{1}{(s+h)^2}\frac{f(z)}{|z|}\ud z\ud s\\
    \le &h^2\int_0^T (s+h)^{\nu-2}\ud s\,.
\end{align*}
We notice that
\begin{align*}
    h^2\int_0^T (s+h)^{\nu-2}\ud s= \begin{cases}
        \frac{h^2}{1-\nu}[h^{\nu-1}-(T+h)^{\nu-1}]\le Ch^{1+\nu},\quad\text{when $0<\nu<1$}\,,\\
        h^2[\log (T+h)-\log(h)]\le C h^{1+\nu^-} \ \text{$ \forall\  0<\nu^-<1$,}\quad\text{when $\nu=1$}\,,\\
        \frac{h^2}{\nu-1}[(T+h)^{\nu-1}-h^{\nu-1}]\le Ch^{2},\quad\text{when $\nu>1$}\,.
    \end{cases}
\end{align*}
Therefore, {for $0\le h<1$, it holds that $R_{4,4}\le C h^{(1+\nu)\wedge 2}$ when $\nu\ne1$ and $R_{4,4}\le C h^{1+\nu^-}$, for any $0<\nu^-<1$  when $\nu=1$}. 
Combining the estimates for $R_{4,i}$ for $i=1,2,3,4$ we get
\begin{equation}\label{E: R_4}
    R_4\le Cp\lr{\sigma (0)^2+L_{\sigma }^2Me^{c\beta t}}h^{{2\mu}}\,,
\end{equation}
where $0<2\mu< 2\mu_0\coloneqq\min\{2\mu_2,\mu_1+1,\nu+1\}$.
Finally, combining all the estimates above for $R_i$, $i = 1,2,3,4$, we obtain the temporal increment estimate
\begin{align*}
    \Norm{u (t',x)-u (t,x)}_p^2\le& C\Bigg[1+b (0)^2+p\sigma (0)^2+\lr{L_b^2+pL_{\sigma }^2}Me^{c\beta t'}\Bigg]|t'-t|^{2{\mu}}\,,
\end{align*}
where $0<2\mu<2\overline{\mu} \coloneqq\min \{2\mu_2,2\overline{\gamma},\nu+1,\overline{\gamma}+\mu_1\}$. The proof of Proposition~\ref{P:TemptInc} is completed by unifying the constants.
\end{proof}

\section{Proof of Theorem~\ref{T:MainTheorem}}\label{S:ProofMainThm}

The proof of Theorem~\ref{T:MainTheorem} is based on a stopping time argument which is adapted from~\cite{mueller:97:long}.
For any fixed $\lr{t,x}\in [0,T]\times \R^3$, we denote the cone $C(t,x)$ as
\begin{align*}
    C(t,x)=\{(s,y)\in [0,t]\times \R^3:|x-y|\le t-s\}\,.
\end{align*}
Without loss of generality, we consider the case when $x=0$. 
For some $K$ large enough which will be chosen later, let $N_n=K2^n$ for all $n\in \bbN$.  Let $b_{N_n}$ and $\sigma_{N_n}$ be the truncation of $b$ and $\sigma$ at level $N_n$, that is,
{
\begin{align}\label{E:TruncatedCoeff}
    b_N(z)=\begin{cases}
        b(z),\quad&|z|\le N\,,\\
        b(N),\quad& z>N\,,\\
        b(-N),\quad& z<-N\,,
    \end{cases}\quad\text{and}\quad\sigma_N(z)=\begin{cases}
        \sigma(z),\quad&|z|\le N\,,\\
        \sigma(N),\quad&z>N\,,\\
        \sigma(-N),\quad& z<-N\,.
    \end{cases}
\end{align}
Since $b$ and $\sigma$ are locally Lipschitz, $b_N$ and $\sigma_N$ are globally Lipschitz. We denote the Lipschitz constants of $b_N$ and $\sigma_N$ by $L_{b_N}$ and $L_{\sigma_N}$, respectively. Under the assumptions of $b$ and $\sigma$ in Theorem~\ref{T:MainTheorem}, we see that
\begin{equation}\label{E:CutoffLip}
    L_{b_N}\le C\lr{\log N}^{\theta_1} \quad\text{and}\quad L_{\sigma_N}\le C\lr{\log N}^{\theta_2}\,,
\end{equation}
for $N$ sufficiently large.}
Let $u_{N}$ be the corresponding solution with drift $b_N$ and diffusion coefficient $\sigma_N$. 
Consider a sequence of stopping times $\tau_n$ defined as
\begin{align}\label{E:Stoppingtimes}
    \tau_n\coloneqq \inf\left\{t\ge 0:\sup_{(s,y)\in C(t,0)}|u_{N_n}(s,y)|\ge N_n\right\}\,.
\end{align}
\begin{lemma}\label{L:StoppingTimeIncreasing}
    Let the stopping times $\tau_n$ be given by~\eqref{E:Stoppingtimes}. Then the sequence $\{\tau_n\}_{n=1}^\infty$ is nondecreasing.
\end{lemma}
\begin{proof}
    For all $(t,x)\in C(\tau_n,0)$, it holds that $|u_{N_n}(t,x)|\le N_n$. From the definition~\eqref{E:TruncatedCoeff}, we have
    \begin{align*}
        b_{N_{n}}(u(t,x))=b_{N_{n+1}}(u(t,x)) \quad\text{and}\quad \sigma_{N_{n}}(u(t,x))=\sigma_{N_{n+1}}(u(t,x)),\quad\text{for all $(t,x)\in C(\tau_n,0)$}\,.
    \end{align*}
    Thus, for all $(t,x)\in C(\tau_n,0)$, since the domain of dependence $C(t,x) $ is contained in $C(\tau_n,0)$, the drift and diffusion coefficients governing the dynamics of $u_{N_n}(t,x)$ and $u_{N_{n+1}}(t,x)$ agree. By the uniqueness of the solution,
    \begin{align*}
        |u_{N_n}(t,x)|=|u_{N_{n+1}}(t,x)|\le N_{n},\quad\text{for all $(t,x)\in C(\tau_n,0)$}\,,
    \end{align*}
    which implies $\tau_{n+1}\ge \tau_n$. The proof is thus complete.
\end{proof}
By virtue of Lemma~\ref{L:StoppingTimeIncreasing}, the limit $\tau_\infty\coloneqq\lim_{n\to\infty}\tau_n$ is well defined. Moreover, by uniqueness and the fact that $b_N$ and $\sigma_N$ coincide with $b$ and $\sigma$ on the set where $|u_{N_n}|\le N_n$, the solution can be defined coherently by setting $u(t,x)=u_{N_n}(t,x)$ for all $t\le \tau_n$. 

Our goal is to show that $\tau_\infty=\infty$ almost surely which enables us to prove that the solution exists for all $(t,x)\in [0,\infty)\times \R^3$. To do this, for all $n$, we define   
\begin{equation}\label{E:Deftn}
    t_n=\frac{1}{K}\sum_{k=1}^n\frac{1}{k}\,,
\end{equation}
with $t_0=0$ and the event $G_n$ to be
\begin{align*}
    G_n\coloneqq\left\{\sup_{(s,y)\in C(t_n,0)}|u_{N_n}(s,y)|<N_n\right\}\,.
\end{align*}
It is clear by definition that
\begin{align}\label{E:GntoTaun}
    G_n\subset \left\{\tau_n> t_n\right\}\,.
\end{align}
We take set intersection on both sides and notice that $\tau_\infty\ge \tau_n$ for all $n\ge 1$ to get
\begin{align}\label{E:GntoGoal}
    \bigcap_{n=1}^\infty G_n\subset \bigcap_{n=1}^\infty\left\{\tau_n> t_n\right\}\subset \bigcap_{n=1}^\infty\left\{\tau_\infty> t_n\right\}= \left\{\tau_\infty=\infty\right\}\,,
\end{align}
where the last equality follows from $\lim_{n\to \infty}t_n=\frac{1}{K}\sum_{k=1}^\infty \frac{1}{k}=\infty$.
Then, to achieve our goal, we only need to show that $\lim_{K\to \infty}\P\lr{\bigcap_{n=1}^\infty G_n}=1$.
Define a sequence of events $E_n$,
\begin{align}\label{E:EventsEn}
    E_n=\left\{\sup_{(t,x)\in C(t_n,0)}\left|u_{N_n}(t,x)\right|\le \sup_{(t,x)\in C(t_{n-1},0)}\left|u_{N_{n-1}}(t,x)\right|+K2^{n-2}\right\}\,,
\end{align}
where $E_1=\left\{\sup_{(t,x)\in C(t_1,0)} |u_{N_1}(t,x)|\le M+\frac{K}{2}\right\}$ and the constant $M$ is given by~\eqref{E:DefM}. For each $n,m$, define the dyadic point set $C_m(t_n,0)$ by
\begin{align}\label{E:CmDyadic}
    C_m(t_n,0)\coloneqq\frac{1}{2^m}\bbZ^4\bigcap C(t_n,0) \,,
\end{align}
where $\bbZ$ is the set of all integers.
For any $k\ge 1$, on the event $\bigcap_{n=1}^k E_n$, we have
\begin{align*}
    \sup_{(t,x)\in C\lr{t_{k},0}}\left|u_{N_k}(t,x)\right|\le& \sup_{(t,x)\in C\lr{t_{k-1},0}}\left|u_{N_{k-1}}(t,x)\right|+K2^{k-2}\\
    \le &M+\frac{K}{2}+K\sum_{i=0}^{k-2}2^i
    =M+\frac{K}{2}\lr{2^k-1}<K2^k=N_k\,,
\end{align*}
which implies that
\begin{align}\label{E:E_ntoG_n}
    \bigcap_{n=1}^k E_n\subset  G_k\,.
\end{align}
Taking intersection on both sides and recall~\eqref{E:GntoGoal}, we obtain
\begin{align*}
    \bigcap_{n=1}^\infty E_n\subset \{\tau_\infty=\infty\}\,.
\end{align*}
Hence, Theorem \ref{T:MainTheorem} is proved if we can establish the following proposition.
\begin{proposition}\label{P:MainProp}
Let $E_n$ be defined in~\eqref{E:EventsEn}. Then, we have that
    \begin{align*}
    \lim_{K\to \infty}\P\lr{\lr{\bigcap_{n=1}^\infty E_n}^c}=0\,.
\end{align*}
\end{proposition}
The proof is deferred to the end of this section, since we still need two additional lemmas. 
For all pairs $(n,m)$ such that $n\ge 1$ and $2^{-m}\le \frac{1}{Kn}$ (i.e. $m\ge \log_2(Kn)$), we denote the set $S_{n,m}$ by 
\begin{align}\label{E:DefSnm}
    \notag S_{n,m}\coloneqq\Big\{(t,x,s,y)\in \R^8:&(t,x)\in C_m(t_n,0)\setminus C_m(t_{n-1},0)\,,\\
    &\quad(s,y)\in C_m(t_{n-1},0)\,,\,|(t,x)-(s,y)|=2^{-m}\Big\}\,,
\end{align}
where $|\cdot|$ refers to the Euclidean norm on $\R^4$. For $\varepsilon>0$ sufficiently small, define the event $F_{n,m}$ as
\begin{align}\label{E:DefFnm}
    F_{n,m}=\left\{\left|u_{N_n}(t,x)-u_{N_{n}}(s,y)\right|\le K2^{n-m\varepsilon},\text{ for all $(t,x,s,y)\in S_{n,m}$}\right\}\,.
\end{align}
Let $F_n$ be the intersection of $F_{n,m}$,
\begin{equation}\label{E:DefFn}
    F_n=\bigcap_{m\ge \log_2 (Kn)}F_{n,m}\,.
\end{equation}
The following lemma gives the relation between $E_n$ and $F_n$.
\begin{lemma}\label{L:E&F}
     For all integers $n$, let $E_n$ be defined in~\eqref{E:EventsEn}. Then we have
\begin{align}
    E_1\bigcap\dots\bigcap E_{n-1}\bigcap F_n\subset E_1\bigcap\dots\bigcap E_{n}\,.
\end{align}
\end{lemma}

\begin{proof}
    On the event $E_1\bigcap\dots\bigcap E_{n-1}\bigcap F_n$, for any $(t,x)\in C\lr{t_n,0}$, choose $(s,y)$ to be the closest point in $C\lr{t_{n-1},0}$ and choose $m_0$ such that $2^{-m_0-2}\le \frac{1}{Kn}< 2^{-m_0-1}$. Without loss of generality, we assume both $(t,x)$ and $(s,y)$ to be dyadic points. Under such setting, we may find a sequence of dyadic points $\{(s_i,y_i)\}_{i=0}^J$ for certain integer $J>0$, such that $(s,y)=(s_0,y_0)$ and $(t,x)=(s_J,y_J)$. Moreover, for each $0\le i\le J-1$, $(s_{i+1},y_{i+1})-(s_i,y_i)$ is one of the forms $(2^{-m},0,0,0)$, $\pm(0,2^{-m},0,0)$, $\pm(0,0,2^{-m},0)$ or $\pm(0,0,0,2^{-m})$, so that $(s_{i+1},y_{i+1},s_i,y_i)\in \bigcup_{m\ge m_0}S_{n,m}$ where $S_{n,m}$ is given by~\eqref{E:DefSnm}. Also, the dyadic sequence $\{(s_i,y_i)\}_{i=0}^J$ can be chosen such that there are at most four of the differences $(s_{i+1},y_{i+1})-(s_i,y_i)$ that are of the magnitude $2^{-m}$. On the event $F_n$, the difference between $u(t,x)$ and $u(s,y)$ is bounded by
    \begin{align*}
        \left|u_{N_n}(t,x)-u_{N_n}(s,y)\right|\le& \sum_{i=0}^J\left|u_{N_n}\lr{s_{i+1},y_{i+1}}-u_{N_n}\lr{s_{i},y_{i}}\right|\\
        \le &C\sum_{m=m_0}^\infty K2^{n-m\varepsilon}\le CK2^n\frac{2^{-m_0\varepsilon}}{1-2^{-\varepsilon}}\\
        \le &CK2^n\frac{1}{1-2^{-\varepsilon}}\frac{1}{\lr{Kn}^\varepsilon}=\frac{C}{1-2^{-\varepsilon}}K^{1-\varepsilon}\frac{2^n}{n^\varepsilon}\,.
    \end{align*}
    Also, on $E_1\bigcap\dots\bigcap E_{n-1}$, by~\eqref{E:E_ntoG_n} and~\eqref{E:GntoTaun}, we have $u_{N_n}(s,y)=u_{N_{n-1}}(s,y)$, for $(s,y)\in C(t_{n-1},0)$. Consequently, for all dyadic points $(t,x)\in \R^4$, it holds that
    \begin{align*}
        \left|u_{N_n}(t,x)\right|\le &\left|u_{N_{n-1}}(s,y)\right|+\left|u_{N_n}(t,x)-u_{N_n}(s,y)\right|\\
        \le & \sup_{(s,y)\in C(t_{n-1},0)}\left|u_{N_{n-1}}(s,y)\right|+\frac{C}{1-2^{-\varepsilon}}K^{1-\varepsilon}\frac{2^n}{n^\varepsilon}\\
        \le &\sup_{(s,y)\in C(t_{n-1},0)}\left|u_{N_{n-1}}(s,y)\right|+K2^{n-2}\,,
    \end{align*}
    which implies $E_n$. The proof is completed since $u_{N_n}$ is continuous almost surely and the dyadic point set is dense in $\R^4$.
\end{proof}
Thanks to Proposition~\ref{P:SpatInc} and Proposition~\ref{P:TemptInc}, we can estimate $\P(F_{n,m}^c)$ for any pair $(n,m)$ such that $m\ge \log _2(Kn)$.
\begin{lemma}\label{L:PFnmc}
For $F_{n,m}$ defined in~\eqref{E:DefFnm}, it holds that for all $p\ge 2$
    \begin{align}\label{E:PFnmc}
        \P\lr{F_{n,m}^c}\le C^p2^{4m}K^{-4}n^{-1}\lr{\log_+ n}^3e^{cp\beta t_n} 2^{-mp(\red{\gamma\wedge\mu})}\lr{K2^{n-m\varepsilon}}^{-p}\,,
    \end{align}
    where
\begin{align*}
    \beta=\sqrt{L_{b_{N_n}}}+p^{\frac{1}{\bar{\nu}+1}}L_{\sigma_{N_n}}^{\frac{2}{\bar{\nu}+1}}\,,\quad 
    t_n=\frac{1}{K}\sum_{k=1}^n\frac{1}{k}\,,
\end{align*}
{$\bar{\nu} = \min\{\nu, \nu_1, \nu_2\}$,}$0<\gamma<\bar{\gamma} = \min\{\alpha_1, \alpha_2, \alpha_3,  \gamma_1, \gamma_2\}$ and $\mu$ is defined in Proposition~\ref{P:TemptInc}. 
\end{lemma}
\begin{proof}
    Given $n,m$ such that $m\ge \log _2(Kn)$, take any $(t,x,s,y)\in S_{n,m}$ defined by~\eqref{E:DefSnm}. By the construction of $C_m(t_n,0)$ and $C_m(t_{n-1},0)$~\eqref{E:CmDyadic}, it holds that either 
    \begin{align*}
        |s-t|=2^{-m},\quad x=y\,,
    \end{align*}
    or
    \begin{align*}
        s=t,\quad |x-y|=2^{-m}\,.
    \end{align*}
    In either case, we apply the moment estimates Proposition~\ref{P:SpatInc} or Proposition~\ref{P:TemptInc} together with Chebyshev's inequality to get
    \begin{align}\label{E:Chebyshev}
        \notag \P\Big(|u_{N_n}(t,x)-u_{N_n}(s,y)|\ge K2^{n-m\varepsilon}\Big)\le& \frac{\E|u_{N_n}(t,x)-u_{N_n}(s,y)|^p}{\lr{K2^{n-m\varepsilon}}^p}\\
        \le& C^pe^{cp\beta t_n} 2^{-mp({\gamma\wedge\mu)}}\lr{K2^{n-m\varepsilon}}^{-p}\,.
    \end{align}
    In order to estimate the cardinality of $S_{n,m}$ defined in~\eqref{E:DefSnm}, we observe that the region $C_m(t_n,0)\setminus C_m(t_{n-1},0)$ has volume of order $t_n^3(t_n-t_{n-1})$. Since $$t_n=\frac1K\sum_{k=1}^n\frac1k \le \frac{C}{K}\log_+ n,\qquad t_n-t_{n-1}=\frac1{Kn}\,,$$
    where $\log_+ n\coloneqq \log(n\vee e)$, it follows that $t_n^3(t_n-t_{n-1})\le C K^{-4}(\log_+ n)^3 n^{-1}$. Also, $\{(t,x),(s,y)\}$ are located at the nearest neighbors of the four dimensional cube with side length $2^{-m}$ and volume $2^{-4m}$. This tells us there are at most $c2^{4m}K^{-4}n^{-1}\lr{\log_+ n}^3$ different choices of $\lr{t,x,s,y}\in S_{n,m}$. By virtue of~\eqref{E:Chebyshev} we obtain
    \begin{align*}
        \P\lr{F_{n,m}^c}=&\P\lr{\bigcup_{(t,x,s,y)\in S_{n,m}}|u_{N_n}(t,x)-u_{N_n}(s,y)|>K2^{n-m\varepsilon}}\\
        \le& \sum_{(t,x,s,y)\in S_{n,m}}\P(|u_{N_n}(t,x)-u_{N_n}(s,y)|\ge K2^{n-m\varepsilon})\\
        \le &C^p2^{4m}K^{-4}n^{-1}\lr{\log_+ n}^3e^{cp\beta t_n} 2^{-mp({\gamma\wedge\mu})}\lr{K2^{n-m\varepsilon}}^{-p}\,,
    \end{align*}
    which is~\eqref{E:PFnmc}.
\end{proof}
Now we are ready to prove Proposition~\ref{P:MainProp}.
\begin{proof}[Proof of Proposition~\ref{P:MainProp}]
Throughout this proof, we omit the subscript $n$ by writing $N=N_n=K2^n$ to avoid redundant subscripts. With the truncation level $N$, we have $L_{b_N}=O\lr{\lr{\log\lr{K2^n}}^{\theta_1}}$ and  $L_{\sigma_N}=O\lr{\lr{\log\lr{K2^n}}^{\theta_2}}$ by our assumption on $b$ and $\sigma$, so $\beta$ in Lemma \ref{L:PFnmc} becomes

\begin{align*}
    \beta=\sqrt{L_{b_N}}+p^{\frac{1}{\bar{\nu}+1}}L_{\sigma_N}^\frac{2}{\bar{\nu}+1}\le C\lr{\lr{\log K+n\log 2}^{\frac{\theta_1}{2}}+p^{\frac{1}{\bar{\nu}+1}}\lr{\log K+n\log 2}^{\frac{2\theta_2}{\bar{\nu}+1}}}\,.
\end{align*}
Also recall that $t_n=\frac{1}{K}\sum_{k=1}^n\frac{1}{k}\le C\frac{\log_+ n}{K}$. By Lemma~\ref{L:PFnmc}, we have
\begin{align}\label{E:PFnmcFull}
    \notag &\P\lr{F_{n,m}^c}
    \le CK^{-4}\lr{\log_+ n}^3n^{-1}2^{4m}\exp\Bigg\{Cp^{\frac{\bar{\nu}+2}{\bar{\nu}+1}}\frac{\log_+ n}{K}\lr{\log K+n\log 2}^{\frac{2\theta_2}{\bar{\nu}+1}}\\
    &+p\Big[-\log K-n\log 2-m\log 2\lr{\lr{{\gamma\wedge\mu}}-\varepsilon}+\frac{C\log_+ n}{K}\lr{\log K+n\log 2}^{\frac{\theta_1}{2}}+C\Big]\Bigg\}\,.
\end{align}
We choose $\varepsilon<\bar{\gamma}\wedge\overline{\mu}$ and denote by $$\zeta\coloneqq \lr{{\gamma\wedge\mu}}-\varepsilon>0\,.$$
By assumption $\theta_1<2$, when $n$ is sufficiently large,
\begin{align*}
    -\log K-n\log 2-m{ \zeta }\log 2+&\frac{\log_+ n}{K}\lr{\log K+n\log 2}^{\frac{\theta_1}{2}}+C\\
    \le &-{C\lr{\log K+n\log 2+m{ \zeta }\log 2}}
\end{align*}
for some constant $C>0$.
We denote
\begin{align*}
    A=C\lr{\log K+n\log 2+m{ \zeta }\log 2}>0\,,
\end{align*}
and
\begin{align}\label{E:DefB}
    B=C \lr{\log K+n\log 2}^{\frac{2\theta_2}{\bar{\nu}+1}}\frac{\log_+ n}{K}\,.
\end{align}
We obtain from~\eqref{E:PFnmcFull}
\begin{align}\label{E:PFnmcAB}
    \P\lr{F_{n,m}^c}\le C\exp\lr{-Ap+Bp^{\frac{\bar{\nu}+2}{\bar{\nu}+1}}}K^{-4}\lr{\log_+ n}^3n^{-1}2^{4m}\,,
\end{align}
for any $p\ge 2$. Let
\begin{align*}
    g(p)=-Ap+Bp^{\frac{\bar{\nu}+2}{\bar{\nu}+1}}\,.
\end{align*}
We optimize the choice of $p$ depending on $n,m$ and $K$ to make $g(p)$ as small as possible.
Taking the derivative of $g(p)$ and setting it equal to zero
\begin{align}\label{E:DerivativeG}
    g'(p)=-A+{\frac{\bar{\nu}+2}{\bar{\nu}+1}}Bp^{\frac{1}{\bar{\nu}+1}}=0\,,
\end{align}
we obtain the solution
\begin{align*}
    p_0=\lr{\frac{A}{B}\frac{\bar{\nu}+1}{\bar{\nu}+2}}^{\bar{\nu}+1}\ge 2
\end{align*}
where the above inequality holds for all sufficiently large $n$, which follows from the assumption $\theta_2<\frac{\overline{\nu}+1}{2}$. Therefore, by taking $p=p_0$, we have
\begin{align*}
    g(p_0)=&-\frac{A^{\bar{\nu}+2}}{B^{\bar{\nu}+1}}\frac{\lr{\bar{\nu}+1}^{\bar{\nu}+1}}{\lr{\bar{\nu}+2}^{\bar{\nu}+2}}\\
    \le &-C\frac{\lr{\log K+n\log 2}^{\bar{\nu}+2}+m^{\bar{\nu}+2}{ \zeta }^{\bar{\nu}+2}}{B^{\bar{\nu}+1}}\frac{\lr{\bar{\nu}+1}^{\bar{\nu}+1}}{\lr{\bar{\nu}+2}^{\bar{\nu}+2}}\\
    \le& -C\frac{\lr{\log K+n\log 2}^{\bar{\nu}+2}}{B^{\bar{\nu}+1}}-C\frac{{ \zeta }^{\bar{\nu}+2}}{B^{\bar{\nu}+1}}m^{\bar{\nu}+2}\,.
\end{align*}
With such choice of $p=p_0$, from~\eqref{E:PFnmcAB} we obtain
\begin{align}\label{E:Fnmch}
    \notag\P\lr{F_{n,m}^c}\le&  C\exp\left\{-C\frac{\lr{\log K+n\log 2}^{\bar{\nu}+2}}{B^{\bar{\nu}+1}}\right\}\\
    &\times\exp\left\{-C\frac{{ \zeta }^{\bar{\nu}+2}}{B^{\bar{\nu}+1}}m^{\bar{\nu}+2}\right\}K^{-4}\lr{\log_+ n}^3n^{-1}2^{4m}\,.
\end{align}
To absorb the factor $2^{4m}$, we make use of the second exponential term. Accordingly, we introduce constants $D$ and $L$ defined by
\begin{align}\label{E:DefD&L}
    D=\frac{C}{2}\frac{{ \zeta }^{\bar{\nu}+2}}{B^{\bar{\nu}+1}}\,,\quad\text{and}\quad L=4\log 2\,.
\end{align}
We have
\begin{align}
    \exp\left\{-\frac{C}{2}\frac{{ \zeta }^{\bar{\nu}+2}}{B^{\bar{\nu}+1}}m^{\bar{\nu}+2}\right\}2^{4m}=\exp\lr{-Dm^{\bar{\nu}+2}+Lm}\,.
\end{align}
Let $h(m)={-Dm^{\bar{\nu}+2}+Lm}$. Differentiating with respect to $m$ yields
\begin{align*}
    h'(m)=-D\lr{\bar{\nu}+2}m^{\bar{\nu}+1}+L\,.
\end{align*}
Let $h'(m_0)=0$ to get $m_0=\lr{\frac{L}{D\lr{\bar{\nu}+2}}}^{\frac{1}{\bar{\nu}+1}}$ and we conclude that for any $m>0$,
\begin{align*}
    h(m)\le h(m_0)=\frac{L^{\frac{\bar{\nu}+2}{\bar{\nu}+1}}{(\bar{\nu}+1)}}{D^{\frac{1}{\bar{\nu}+1}}\lr{\bar{\nu}+2}^{\frac{\bar{\nu}+2}{\bar{\nu}+1}}}\le C\frac{B}{{{ \zeta }^{\frac{\bar{\nu}+2}{\bar{\nu}+1}}}}\,.
\end{align*}
Now we can bound $\P\lr{F_{n,m}^c}$ from~\eqref{E:Fnmch} by
\begin{align*}
    \P\lr{F_{n,m}^c}\leq & C\exp\left\{-C\frac{\lr{\log K+n\log 2}^{\bar{\nu}+2}}{B^{\bar{\nu}+1}}\right\}\exp\left\{-\frac{C}{2}\frac{{ \zeta }^{\bar{\nu}+2}}{B^{\bar{\nu}+1}}m^{\bar{\nu}+2}\right\}\\
    &\quad\quad\times K^{-4}\lr{\log_+ n}^3n^{-1}\exp\lr{ h(m)}\\
    \le &C\exp\left\{-C\frac{\lr{\log K+n\log 2}^{\bar{\nu}+2}}{B^{\bar{\nu}+1}}\right\}K^{-4}\lr{\log_+ n}^3n^{-1} \exp\left\{\frac{B}{{{ \zeta }^{\frac{\bar{\nu}+2}{\bar{\nu}+1}}}}\right\}\\
    &\quad\quad \times \exp\left\{-\frac{C}{2}\frac{{ \zeta }^{\bar{\nu}+2}}{B^{\bar{\nu}+1}}m^{\bar{\nu}+2}\right\}\,.
\end{align*}
By \eqref{E:DefFn}, $F_n=\bigcap_{m\ge \log_2 (Kn)}F_{n,m}$, 
we have
\begin{equation}\label{E: r_n(K)}
\begin{aligned}
    \P\lr{F_n^c}=&\P\lr{\bigcup_{m=\log_2(Kn)}^\infty F_{n,m}^c}\le \sum_{m=\lceil\log_2(Kn)\rceil-2}^\infty\P \lr{F_{n,m}^c}\\
    \le &C\exp\left\{-C\frac{\lr{\log K+n\log 2}^{\bar{\nu}+2}}{B^{\bar{\nu}+1}}\right\}K^{-4}\lr{\log_+ n}^3n^{-1} \exp\left\{\frac{B}{{{ \zeta }^{\frac{\bar{\nu}+2}{\bar{\nu}+1}}}}\right\}\\
    &\quad\quad\times \int_0^\infty \exp\left\{-\frac{C}{2}\frac{{ \zeta }^{\bar{\nu}+2}}{B^{\bar{\nu}+1}}z^{\bar{\nu}+2}\right\}\ud z\\
    \le& C\exp\left\{-C\frac{\lr{\log K+n\log 2}^{\bar{\nu}+2}}{B^{\bar{\nu}+1}}+\frac{B}{{{ \zeta }^{\frac{\bar{\nu}+2}{\bar{\nu}+1}}}}+ \frac{\bar{\nu}+1}{\bar{\nu}+2} \log B\right\}K^{-4}\lr{\log_+ n}^3n^{-1} \zeta^{-1}\\
    =:& r_n(K)\,,
\end{aligned} 
\end{equation}
where $B$ is given by~\eqref{E:DefB}. Let
\begin{align*}
    L_1=&-C\frac{\lr{\log K+n\log 2}^{\bar{\nu}+2}}{B^{\bar{\nu}+1}} \\
    =& -C \left(\log K + n \log 2\right)^{\bar{\nu}+2-2\theta_2} K ^{\bar{\nu}+1} \left(\log_+ n\right)^{-(\bar{\nu}+1)}\,,
\end{align*}    
and
\begin{align*}
     L_2=&\frac{B}{{{ \zeta }^{\frac{\bar{\nu}+2}{\bar{\nu}+1}}}}
     = C \left(\log K + n \log 2\right)^{\frac{2\theta_2}{\bar{\nu}+1}} K^{-1} \left(\log_+ n\right) \zeta^{-\frac{\bar{\nu}+2}{\bar{\nu}+1}}\,, 
\end{align*}
we see that the order of $n $ in $L_1$ is $\bar{\nu}+2-2\theta_2$ and the order of $n$ in $L_2$ is $\frac{2\theta_2}{\bar{\nu}+1}$. $\P\lr{F_n^c}$ is summable with respect to $n$ when $\bar{\nu}+2-{2\theta_2}>\frac{2\theta_2}{\bar{\nu}+1}$, that is $\theta_2<\frac{\bar{\nu}+1}{2}$ as in the assumption. Thus it holds that
\begin{align*}
    \P\left( \left[ \bigcap_{n=1}^{\infty} E_n \right]^c \right)=& \P\lr{\bigcup_{n=1}^\infty E_n^c}
        = \P\lr{\bigcup_{n=1}^\infty\Bigl[\bigcup_{m=1}^nE_m^c\setminus \bigcup_{m=1}^{n-1}E_m^c \Bigr]}
        \leq \sum_{n=1}^{\infty} \P(E_n^c \cap E_{n-1} \cap \cdots \cap E_1) \\
    \leq& \sum_{n=1}^{\infty} \P(F_n^c\cap E_{n-1} \cap \cdots \cap E_1) 
    \leq \sum_{n=1}^\infty \P\lr{F_n^c}<\infty\,,
\end{align*}
where we applied Lemma~\ref{L:E&F} in the second inequality.
Moreover, the $ r_n(K)$ defined as in \eqref{E: r_n(K)}
can be expressed in more detail as
\begin{equation*}
    \begin{aligned}
r_n(K) =& C \exp \bigg\{-C\left(\log K + n \log 2\right)^{\bar{\nu}+2-2\theta_2}(\log_+ n)^{-(\bar{\nu}+1)}K^{\bar{\nu}+1}\\
&\ \ + C \left(\log K + n \log 2\right)^{\frac{2\theta_2}{\bar{\nu}+1}}\zeta^{-\frac{\bar{\nu}+2}{\bar{\nu}+1}}+\frac{\bar{\nu}+1}{\bar{\nu}+2}\log \left( C \left(\log K + n \log 2\right)^{\frac{2\theta_2}{\bar{\nu}+1}}\frac{\log_+ n}{K}\right)\\
&\qquad -4\log K\bigg\} \left(\log_+ n\right)^3 n^{-1}\zeta^{-1}\\
&\qquad \eqqcolon Ce^{\kappa_n (K)}\left(\log_+ n\right)^3 n^{-1}\zeta^{-1}\,.
\end{aligned}
\end{equation*}
The exponential $\kappa_n(K)$ can be bound uniformly in $K$. Indeed, since $\theta_2<\frac{\bar{\nu}+1}{2}$, we have $\bar\nu+2-2\theta_2>1$ and $\frac{2\theta_2}{\bar{\nu}+1}<1$. Thus,
\begin{align*}
    \kappa_n(K)\le& -C\frac{n^{\bar\nu+2-2\theta_2}}{\lr{\log_+n}^{\bar\nu+1}}K^{\bar\nu+1}+C\lr{\log K}^{\frac{2\theta_2}{\bar\nu+1}}+C n^{\frac{2\theta_2}{\bar\nu+1}}+C\\
    &+C\log\log K+C\log n+C\log\log_+n-C\log K\\
    \le & -C n K^{\bar \nu+1}+C n^{\frac{2\theta_2}{\bar\nu+1}}+C\\
    \le &-C n+C\,.
\end{align*}
Hence, we obtain
\begin{align*}
    \sum_{n=1}^\infty r_n(K)\le C\sum_{n=1}^\infty e^{-Cn}\left(\log_+ n\right)^3 n^{-1}<\infty\,.
\end{align*}
By Dominated Convergence Theorem,
\begin{align*}
    \lim_{K\to \infty}\P\left( \left[ \bigcap_{n=1}^{\infty} E_n \right]^c \right)\le \sum_{n=1}^\infty \lim_{K\to \infty}\P\lr{F_n^c}\le \sum_{n=1}^\infty \lim_{K\to \infty} r_n(K)=0\,,
\end{align*}
which finishes the proof of Proposition~\ref{P:MainProp}.
\end{proof}

The proof of Theorem \ref{T:MainTheorem} is now completed. 

\section{Examples}\label{S:Examples}
In this section, we provide some examples of covariance functions that satisfy Assumption~\ref{A:CovAssumption}.
\subsection{Functions in $C_b^2(\R^3)$}
\begin{proposition}
    Let $C^2_b(\R^3)$ be the space of bounded functions, with bounded continuous derivatives up to order $2$. Then, any $f\in C_b^2(\R^3)$ satisfies Assumption~\ref{A:CovAssumption} with $\nu=\nu_1=\nu_2=2$ and $\gamma_1=\gamma_2=\mu_1=\mu_2=1$.
\end{proposition}
\begin{proof}
    Here we verify~\eqref{E:ZeroOrderDifference} through~\eqref{E:SecondOrderDifference_2} one by one. Since $f$ is bounded,
    \begin{align*}
        \int_{|z|\le t}\frac{f(z)}{|z|}\ud z\le \|f\|_\infty\int_{|z|\le t}\frac{1}{|z|}\ud z \le Ct^2\,.
    \end{align*}
    Thus,~\eqref{E:ZeroOrderDifference} is verified with $\nu=2$. For~\eqref{E:FirstOrderDifference}, we apply the Mean Value Theorem,
    \begin{align*}
    \int_{|z|\le t}\frac{|f(z+w)-f(z)|}{|z|}\ud z\le& \int_{|z|\le t}\int_0^1\frac{|\nabla f(z+\theta w)||w|}{|z|}\ud \theta \ud z\\
    \le &\|\nabla f\|_\infty|w|\int_{|z|\le t}\frac{1}{|z|}\ud z\le C|w|t^2\,,
\end{align*}
so~\eqref{E:FirstOrderDifference} is verified with $\gamma_1=1$ and $\nu_1=2$. For~\eqref{E:SecondOrderDifference}, also from the Mean Value Theorem and the boundedness of the second derivative, we obtain
\begin{align*}
    &\int_{|z| \leq t} \frac{|f(z + w) + f(z - w) - 2f(z)|}{|z|} \, dz\\
    \le&  \int_{|z|\leq t}   \|\nabla^2 f\|_\infty \frac{1}{|z|} |w|^2 dz \leq C |w|^2t^2\,.
\end{align*}
Therefore,~\eqref{E:SecondOrderDifference} holds for $\gamma_2=1$ and $\nu_2=2$.
For~\eqref{E:FirstOrderDifference_2} and~\eqref{E:SecondOrderDifference_2}, similarly from the Mean Value Theorem, we have
\begin{align*}
    &\int_0^T\int_{S^2\times S^2} s\left|f\lr{(s+h)\xi-(s+h)\eta}-f\lr{(s\xi-(s+h)\eta)}\right|A(\ud \xi )A(\ud \eta)\ud s\\
    &\quad\quad\le  C\int_0^t \|\nabla f\|_\infty sh\ud s\le C h\,,
\end{align*}
and
\begin{align*}
        &\int_0^T\int_{S^2\times S^2}s^2\bigg|f\lr{(s+h)\xi+(s+h)\eta}-f\lr{s\xi+(s+h)\eta}\\
    &\quad-f\lr{(s+h)\xi+s\eta}+f(s\xi+s\eta)\bigg|A(\ud \xi)A(\ud \eta)\ud s\le C\|\nabla^2 f\|_\infty h^2\int_0^t s^2\ud s  \le C\|\nabla^2 f\|_\infty h^2\,.
    \end{align*}
So,~\eqref{E:FirstOrderDifference_2} and~\eqref{E:SecondOrderDifference_2} hold for $\mu_1=\mu_2=1$. The proof is thus completed.
\end{proof}

\subsection{Riesz kernel}
\begin{proposition}
    For $0<\beta<2$, let $f$ be the Riesz kernel.
    \begin{align}\label{E:Riesz kernel}
        f(z)=|z|^{-\beta}\,.
    \end{align}
    Then, for any $0<\gamma_1,\mu_1<\min\{2-\beta,1\}$ and $0<\gamma_2,\mu_2<1-\beta/2$, $f$ satisfies Assumption~\ref{A:CovAssumption} with parameters
$\nu=2-\beta$, $\nu_1=2-\beta-\gamma_1$, $\nu_2=2-\beta-2\gamma_2$, and the corresponding choices of $\gamma_1$, $\gamma_2$, $\mu_1$, and $\mu_2$.
\end{proposition}
\begin{proof}
    Here we only prove~\eqref{E:FirstOrderDifference} and~\eqref{E:SecondOrderDifference}. For~\eqref{E:ZeroOrderDifference},~\eqref{E:FirstOrderDifference_2} and~\eqref{E:SecondOrderDifference_2}, we refer readers to~\cite[Proposition 5.3]{hu.huang.ea:14:on}.
    To prove~\eqref{E:FirstOrderDifference}, for any $w\in \R^3$, let $\xi=w/|w|$ be the corresponding unit vector and denote $h=|w|$. Noting that $\nu_1+1+\gamma_1=3-\beta$, we apply Lemma~\ref{L:RieszHolder} (see Eq.\eqref{E_:RieszDiff}) with $a=\nu_1+1$ and $b=\gamma_1$ to see that
\begin{equation}\label{E_:RieszFirst}
    \begin{aligned}
        \int_{|u|\le t}&\frac{\left||u+h\xi|^{-\beta}-|u|^{-\beta}\right|}{|u|}\ud u\\
        &\le  h ^{\gamma_1}\int_{|u|\le t}\ud u\int_{\R^3}\ud  v\frac{|u-h v|^{\nu_1-2}}{|u|}\left|| v+\xi|^{\gamma_1-3}-| v|^{\gamma_1-3}\right|\\
    &\le  h ^{\gamma_1}\int_{|u|\le t}\ud u\int_{| v|\le 3}\ud  v \frac{|u-h v|^{\nu_1-2}}{|u|}\lr{| v+\xi|^{\gamma_1-3}+| v|^{\gamma_1-3}}\\
    &\quad+ h ^{\gamma_1}\int_{|u|\le t}\ud u\int_{| v|> 3}\ud  v \frac{|u-h v|^{\nu_1-2}}{|u|}\left|| v+\xi|^{\gamma_1-3}-| v|^{\gamma_1-3}\right|\\
    &\eqqcolon A_1+A_2
    \end{aligned}
\end{equation}
For $A_1$, we apply Lemma~\ref{L:RearrangementIneq} to obtain
\begin{equation}\label{E_:RieszSmall_1}
    \begin{aligned}
         h ^{\gamma_1}&\int_{|u|\le t}\ud u\int_{| v|\le 3}\ud  v \frac{|u-h v|^{\nu_1-2}}{|u|}| v+\xi|^{\gamma_1-3}\\
        &\le  h ^{\gamma_1}\int_{|u|\le t}\ud u\int_{| v|\le 4}\ud  v \frac{|u-h v+h\xi|^{\nu_1-2}}{|u|}| v|^{\gamma_1-3}\\
    &\le  h ^{\gamma_1}\int_{|u|\le t}\ud u\int_{| v|\le 4}\ud  v \frac{|u|^{\nu_1-2}}{|u|}| v|^{\gamma_1-3}\\
    &\le C h ^{\gamma_1}t^{\nu_1}\,,
    \end{aligned}
\end{equation}
and
\begin{equation}\label{E_:RieszSmall_2}
    \begin{aligned}
         h ^{\gamma_1}\int_{|u|\le t}\ud u\int_{| v|\le 3}\ud  v \frac{|u-h v|^{\nu_1-2}}{|u|}| v|^{\gamma_1-3}\le& 
      h ^{\gamma_1}\int_{|u|\le t}\ud u\int_{| v|\le 3}\ud  v \frac{|u|^{\nu_1-2}}{|u|}| v|^{\gamma_1-3}\\
    \le& C h ^{\gamma_1}t^{\nu_1}\,.
    \end{aligned}
\end{equation}
So $A_1 \leq C h^{\gamma_1}t^{\nu_1}$. 

For $A_2$, from the Mean Value Theorem, we see that $\left||v+\xi|^{\gamma_1-3}-|v|^{\gamma_1-3}\right|\le C|v|^{\gamma_1-4}$ for $|v|>3$. Applying Lemma~\ref{L:RearrangementIneq} again, we see that
\begin{equation}\label{E_:RieszBig}
    \begin{aligned}
         h ^{\gamma_1}&\int_{|u|\le t}\ud u\int_{| v|> 3}\ud  v \frac{|u-h v|^{\nu_1-2}}{|u|}\left|| v+\xi|^{\gamma_1-3}-| v|^{\gamma_1-3}\right|\\
        \le &C h ^{\gamma_1}\int_{|u|\le t}\ud u\int_{| v|> 3}\ud  v \frac{|u-h v|^{\nu_1-2}}{|u|}| v|^{\gamma_1-4}\\
    \le &C h ^{\gamma_1}\int_{|u|\le t}\ud u\int_{| v|> 3}\ud  v \frac{|u|^{\nu_1-2}}{|u|}| v|^{\gamma_1-4}\\
    \le & C h ^{\gamma_1}t^{\nu_1}\,.
    \end{aligned}
\end{equation}
where the last line holds since $\gamma_1-4<-3$.
A combination of~\eqref{E_:RieszSmall_1},~\eqref{E_:RieszSmall_2},~\eqref{E_:RieszBig}, together with~\eqref{E_:RieszFirst} shows
\begin{align}\label{E_:RieszFirstResult}
        \int_{|u|\le t}\frac{\left||u+h\xi|^{-\beta}-|u|^{-\beta}\right|}{|u|}\ud u \le Ct^{\nu_1} h ^{\gamma_1}\,,
\end{align}
which verifies~\eqref{E:FirstOrderDifference}.

For \eqref{E:SecondOrderDifference}, we apply Lemma~\ref{L:RieszHolder} (see Eq.\eqref{E_:RieszDiff_2}) with $a=\nu_2+1$ and $b=2\gamma_2$ to find that
\begin{align*}
    &\int_{|u|\le t}\frac{\left|f(u+h\xi)+f(u-h\xi)-2f(u)\right|}{|u|}\ud u\\
    &\qquad\le
     h ^{2\gamma_2}\int_{|u|\le t}\int_{\R^3} \frac{|u-h v|^{\nu_2-2}}{|u|}\left|| v+\xi|^{2\gamma_2-3}+| v-\xi|^{2\gamma_2-3}-2| v|^{2\gamma_2-3}\right|\ud  v\ud u\\
    &\qquad\le  h ^{2\gamma_2}\int_{|u|\le t}\int_{|v|\le 3} \frac{|u-h v|^{\nu_2-2}}{|u|}\left|| v+\xi|^{2\gamma_2-3}+| v-\xi|^{2\gamma_2-3}-2| v|^{2\gamma_2-3}\right|\ud  v\ud u\\
    &\qquad\quad + h ^{2\gamma_2}\int_{|u|\le t}\int_{|v|>3} \frac{|u-h v|^{\nu_2-2}}{|u|}\left|| v+\xi|^{2\gamma_2-3}+| v-\xi|^{2\gamma_2-3}-2| v|^{2\gamma_2-3}\right|\ud  v\ud u\\
    &\qquad\eqqcolon B_1+B_2\,.
\end{align*}
Here we have followed the same idea as in~\eqref{E_:RieszFirst}. For $B_1$,
\begin{align*}
    B_1=& h ^{2\gamma_2}\int_{|u|\le t}\int_{| v|\le 3} \frac{|u-h v|^{\nu_2-2}}{|u|}\left|| v+\xi|^{2\gamma_2-3}+| v-\xi|^{2\gamma_2-3}-2| v|^{2\gamma_2-3}\right|\ud  v\ud u\\
    \le & h ^{2\gamma_2}\int_{|u|\le t}\int_{| v|\le 3} \frac{|u-h v|^{\nu_2-2}}{|u|}\left(| v+\xi|^{2\gamma_2-3}+| v-\xi|^{2\gamma_2-3}+2| v|^{2\gamma_2-3}\right)\ud  v\ud u\,.
\end{align*}
For each term, similar to~\eqref{E_:RieszSmall_1} and~\eqref{E_:RieszSmall_2} we get
\begin{align}\label{E:K1}
    B_1\le C h ^{2\gamma_2}t^{\nu_2}\,.
\end{align}
For $B_2$, since $2\gamma_2-5<-3$, we apply the Mean Value Theorem to get
\begin{align}\label{E_:K2}
     \notag B_2=& h ^{2\gamma_2}\int_{|u|\le t}\int_{| v|> 3} \frac{|u-h v|^{\nu_2-2}}{|u|}\left|| v+\xi|^{2\gamma_2-3}+| v-\xi|^{2\gamma_2-3}-2| v|^{2\gamma_2-3}\right|\ud  v\ud u\\
     \notag\le &C h ^{2\gamma_2}\int_{|u|\le t}\ud u\int_{| v|> 3}\ud  v \frac{|u-h v|^{\nu_2-2}}{|u|}| v|^{2\gamma_2-5}\\
    \notag\le &C h ^{2\gamma_2}\int_{|u|\le t}\ud u\int_{| v|> 3}\ud  v \frac{|u|^{\nu_2-2}}{|u|}| v|^{2\gamma_2-5}\\
    \le & C h ^{2\gamma_2}t^{\nu_2}\,.
\end{align}
A combination~\eqref{E:K1} and~\eqref{E_:K2} shows that
\begin{align*}
    \int_{|u|\le t}\frac{\left|f(u+h\xi)+f(u-h\xi)-2f(u)\right|}{|u|}\ud u\le  C h ^{2\gamma_2}t^{\nu_2}\,,
\end{align*}
which verifies~\eqref{E:SecondOrderDifference}. The proof is thus completed.
\end{proof}


\subsection{Bessel kernel}
\begin{proposition}
    Let $f(x)$ be the Bessel kernel of order $\alpha>3$, defined by $$f(x)=\int_0^\infty v^{\frac{\alpha-5}{2}}e^{-v}e^{\frac{-|x|^2}{4v}}\ud v.$$ Then, $f(x)$ satisfies Assumption~\ref{A:CovAssumption} with any $0<\nu<1$, $\nu_1=\nu_2=2$, $0<\gamma_1,\mu_1<\min\lr{\alpha-3,1}$ and $0<\gamma_2,\mu_2<\min\lr{\alpha-3,2}$.
\end{proposition}
\begin{proof}
    As in the Riesz kernel case, we refer the proof of~\eqref{E:ZeroOrderDifference},~\eqref{E:FirstOrderDifference_2} and~\eqref{E:SecondOrderDifference_2} to~\cite[Proposition 5.4]{hu.huang.ea:14:on}. We adapt the proof of the cited proposition to show~\eqref{E:FirstOrderDifference} and~\eqref{E:SecondOrderDifference}. We apply $|e^{-a}-e^{-b}|\le |a-b|^\gamma\lr{e^{-a}\vee e^{-b}}$, for any $a,b\ge 0$ and $0\le \gamma\le 1$ to get
\begin{equation}
\begin{aligned}
    \left| e^{-\frac{|z+w|^2}{4 v}} - e^{-\frac{|z|^2}{4 v}} \right| \leq& \left| \frac{1}{4 v} \right|^{\gamma} | |z+w|^2 - |z|^2 |^{\gamma} \left( e^{-\frac{|z+w|^2}{4 v}} \vee e^{-\frac{|z|^2}{4 v}} \right) \\
    \leq& C |w|^{\gamma} (|z+w|^{\gamma} + |z|^{\gamma}) \frac{1}{ v^{\gamma}} \left( e^{-\frac{|z+w|^2}{4 v}} + e^{-\frac{|z|^2}{4 v}} \right).
\end{aligned}
\end{equation}
 Consequently,
\begin{align}\label{E_:BesselSecond}
    \int_{|z|\le t}\frac{\left|f(z+w)-f(z)\right|}{|z|}\ud z\le |w|^{\gamma}\int_0^\infty v^{\frac{\alpha-5}{2}-\gamma}e^{-v}\lr{I(t,w,v)+J(t,w,v)}\ud v\,,
\end{align}
where
\begin{equation}\label{E_:EstimateI}
\begin{aligned}
    I(t,w,v)=&\int_{|z|\le t}\lr{|z+w|^\gamma+|z|^\gamma}\frac{e^{-\frac{|z+w|^2}{4v}}}{|z|}\ud z\\
    =&\int_{|z|\le t}{|z+w|^\gamma}\frac{e^{-\frac{|z+w|^2}{4v}}}{|z|}\ud z+\int_{|z|\le t}{|z|^\gamma}\frac{e^{-\frac{|z+w|^2}{4v}}}{|z|}\ud z\\
    \eqqcolon& I_1(t,w,v)+I_2(t,w,v)\,,
\end{aligned}
\end{equation}
and
\begin{align*}
    J(t,w,v)=&\int_{|z|\le t}\lr{|z+w|^\gamma+|z|^\gamma}\frac{e^{-\frac{|z|^2}{4v}}}{|z|}\ud z\\
    =&\int_{|z|\le t}{|z+w|^\gamma}\frac{e^{-\frac{|z|^2}{4v}}}{|z|}\ud z+\int_{|z|\le t}{|z|^\gamma}\frac{e^{-\frac{|z|^2}{4v}}}{|z|}\ud z\\
    \eqqcolon& J_1(t,w,v)+J_2(t,w,v)\,.
\end{align*}
For $I_1(t,w,v)$, noting that $|x|^\gamma e^{-\frac{|x|^2}{4}}\le Ce^{-\frac{|x|^2}{8}}$, by the change of variable $x=\frac{z}{\sqrt{v}}$, we have
\begin{align}\label{E_:EstimateI_1}
    \notag I_1(t,w,v)\le& C\int_{|x|\le \frac{t}{\sqrt{v}}}v^{\frac{\gamma+2}{2}}\frac{1}{|x|}e^{-\frac{\left|x+\frac{w}{\sqrt{v}}\right|^2}{8}}\ud x\\
    \le &C\int_{|x|\le \frac{t}{\sqrt{v}}}v^{\frac{\gamma+2}{2}}\frac{1}{|x|}\ud x\le Cv^{\frac{\gamma}{2}}t^2\,.
\end{align}
For $I_2(t,w,v)$, on the one hand, we bound the exponential functions by $1$ to see that
\begin{align*}
    I_2(t,w,v)\le& \int_{|z|\le t}|z|^{\gamma-1}\ud z\le C t^{\gamma+2}\,.
\end{align*}
On the other hand, with the change of variable $\frac{z+w}{\sqrt{v}}=x$ we observe that
\begin{align*}
    I_2(t,w,v)= &\int_{\left|x-\frac{w}{\sqrt{v}}\right|\le \frac{t}{\sqrt{v}}}{v^{\frac{\gamma+2}{2}}\left|x-\frac{w}{\sqrt{v}}\right|^{\gamma-1}}e^{-\frac{|x|^2}{4}}\ud x\\
    \le &v^{\frac{\gamma+2}{2}}\int_{\R^3}{\left|x-\frac{w}{\sqrt{v}}\right|^{\gamma-1}}e^{-\frac{|x|^2}{4}}\ud x\,.
\end{align*}
The integral above is bounded uniformly in $w$ and $v$ by Lemma~\ref{L:Rearrangement_2}. Consequently, we obtain
\begin{align*}
    I_2(t,w,v)\le Cv^{\frac{\gamma+2}{2}}\,,
\end{align*}
where the constant $C$ is independent of $t$.
By interpolation, for all $0\le \theta\le 1$
\begin{align*}
    I_2(t,w,v)\le Ct^{\lr{\gamma+2}\theta}v^{\lr{\frac{\gamma+2}{2}}(1-\theta)}\,.
\end{align*}
We can take $\theta=\frac{2}{\gamma+2}$ to get
\begin{align}\label{E_:EstimateI_2}
    I_2(t,w,v)\le Ct^2v^{\frac{\gamma}{2}}\,.
\end{align}
By combining~\eqref{E_:EstimateI},~\eqref{E_:EstimateI_1} and~\eqref{E_:EstimateI_2}, we see that
\begin{align*}
    I(t,w,v)\le Ct^2v^{\frac{\gamma}{2}}\,.
\end{align*}
Similarly,
\begin{align*}
    J(t,w,v)\le Ct^2v^{\frac{\gamma}{2}}\,.
\end{align*}
Consequently, when $0<\gamma<\alpha-3$, we can bound~\eqref{E_:BesselSecond} by
\begin{align*}
    \int_{|z|\le t}\frac{\left|f(z+w)-f(z)\right|}{|z|}\ud z\le C|w|^\gamma t^2\int_0^\infty v^{\frac{\alpha-\gamma-5}{2}}e^{-v}\ud v\le C|w|^\gamma t^2\,.
\end{align*}
 So,~\eqref{E:FirstOrderDifference} holds for all $0<\gamma<\min\lr{1,\alpha-3}$.
According to~\cite[the last line of page 388]{hu.huang.ea:14:on}, for any $0\le \gamma'\le 2$,
\begin{align*}
    \int_{|z|\le t}&\frac{\left|f(z+w)+f(z-w)-2f(z)\right|}{|z|}\ud z\le C|w|^{\gamma'}\int_0^\infty\ v^{\frac{\alpha-5-\gamma'}{2}}e^{-v}K(t,v,w)\ud v\,,
\end{align*}
where
\begin{align*}
    K(t,v,w)\le& \int_{|z|\le t}\int_0^1\ud \lambda\int_0^1\ud \mu\lr{e^{-\frac{|z-\lr{\lambda-\mu}w|^2}{8v}}+e^{-\frac{|z+w|^2}{4v}}+e^{-\frac{|z-w|^2}{4v}}+2e^{-\frac{|z|^2}{4v}}}\frac{1}{|z|}\ud z\,.
\end{align*}
We can bound the exponential functions by $1$ to get
\begin{align*}
    K(t,v,w)\le Ct^2\,.
\end{align*}
As a result,
\begin{align*}
    \int_{|z|\le t}&\frac{\left|f(z+w)+f(z-w)-2f(z)\right|}{|z|}\ud z\le C|w|^{\gamma'}t^{2}\int_0^\infty v^{\frac{\alpha-5-\gamma'}{2}}e^{-v}\ud v\le C|w|^{\gamma'}t^{2}\,,
\end{align*}
for any $0<\gamma'<\min\lr{\alpha-3,2}$. This completes the proof.
\end{proof}

\appendix
\section{Auxiliary Lemmas} 

The following lemma, taken from~\cite[Lemma 2.6]{dalang.sanz-sole:09:holder-sobolev}, allows us to verify that the Riesz kernel satisfies Assumption~\ref{A:CovAssumption}.
\begin{lemma}\label{L:RieszHolder}
    Let $f(u)=|u|^{-\beta}$ be the Riesz kernel with $0<\beta<3$. For all positive parameters $a,b$ such that $a+b=3-\beta$, the following statements hold.
\begin{enumerate}
    \item For all $u,\xi\in\mathbb{R}^3$ and $h\in\mathbb{R}$,
    \begin{align}\label{E_:RieszDiff}
        f(u+h\xi)-f(u)
        =  |h| ^b \int_{\mathbb{R}^3} |u-hw|^{a-3}
        \bigl(|w+\xi|^{\,b-3}-|w|^{\,b-3}\bigr)\,\mathrm{d}w .
    \end{align}

    \item If, in addition, $|\xi|=1$, then
    \begin{align}\label{E_:RieszDiff_2}
        \bigl|f(u+h\xi)+f(u-h\xi)-2f(u)\bigr|
        \le  |h| ^b \int_{\mathbb{R}^3}& |u-hw|^{a-3}\\
        &\notag \times \bigl||w+\xi|^{\,b-3}+|w-\xi|^{\,b-3}-2|w|^{\,b-3}\bigr|\,\mathrm{d}w .
    \end{align}
\end{enumerate}
\end{lemma}
The following two lemmas establish upper bounds on the integrals of the Riesz kernel with respect to two different measures.
\begin{lemma}\label{L:RearrangementIneq}
For all $r\ge 0$ and $\alpha\in (-2,0)$, it holds for all $v\in \R^3$ that
    \begin{align*}
        \int_{|u|\le r}\frac{|u+v|^{\alpha}}{|u|}\ud u\le \int_{|u|\le r}|u|^{\alpha-1}\ud u \,.
    \end{align*}
\end{lemma}
\begin{proof}
    By Lemma~\ref{L:KernelConv}, we have
    \begin{align*}
         \int_{|u|\le r}\frac{|u+v|^{\alpha}}{|u|}\ud u=&8\pi \int _{\R^3} |u+v|^{\alpha} (G_{r/2}\ast G_{r/2})(\ud u)\,.
    \end{align*}
    Let $\widehat{G}_{r/2}(\xi)$ be the Fourier transform of the probability measure $G_{r/2}(\ud u)$, that is,
    \begin{align*}
        \widehat{G}_{r/2}(\xi)=\int_{\R^3}e^{-i\xi \cdot u}G_{r/2}(\ud u)\,.
    \end{align*}
    We apply~\cite[Corollary 3.4]{foondun.khoshnevisan:13:on} with $\mu(\ud u)=G_{r/2}(\ud u)$ to get
    \begin{align*}
        &8\pi \int _{\R^3} |u+v|^{\alpha} (G_{r/2}\ast G_{r/2})(\ud u)\\
        =& \frac{C_\alpha}{\pi^2} \int_{\R^3}\left|\widehat{G}_{r/2}(\xi )\right|^2|\xi|^{-\alpha-3} e^{i v \cdot \xi} \ud \xi
        \le \frac{C_\alpha}{\pi^2} \int_{\R^3}\left|\widehat{G}_{r/2}(\xi )\right|^2|\xi|^{-\alpha-3}  \ud \xi\\
        = &8\pi \int _{\R^3} |u|^{\alpha} (G_{r/2}\ast G_{r/2})(\ud u)\\
        =&\int_{|u|\le r}|u|^{\alpha-1}\ud u\,,
    \end{align*}
    where $C_\alpha|\xi|^{-\alpha-3}$ is the Fourier transform of $f(u)=|u|^{\alpha}$. The proof is thus completed.
\end{proof}
\begin{lemma}\label{L:Rearrangement_2}
Let $\alpha\in(-3,0)$ and $\beta>0$. Then there exists a constant
$C=C(\alpha,\beta)>0$ such that
    \begin{align*}
    \sup_{y\in \R^3}\int_{\R^3}{\left|x-y\right|^{\alpha}}e^{-\frac{|x|^2}{\beta}}\ud x\le C\,.
    \end{align*}
In particular, the integral is finite and uniformly bounded with respect to the spatial shift $y\in\R^3$.
\end{lemma}
\begin{proof}
By Plancherel's theorem, for $\alpha\in (-3,0)$, we have
    \begin{align*}
    \int_{\R^3}{\left|x-y\right|^{\alpha}}e^{-\frac{|x|^2}{\beta}}\ud x=&\frac{C_{\alpha}}{(2\pi)^3}\int_{\R^3}|\xi|^{-\alpha-3}\lr{\frac{\beta}{2}}^{\frac{3}{2}}e^{-\frac{\beta |\xi|^2}{4}}e^{-iy\cdot \xi}\ud \xi\\
    \le &\frac{C_{\alpha}}{(2\pi)^3}\int_{\R^3}|\xi|^{-\alpha-3}\lr{\frac{\beta}{2}}^{\frac{3}{2}}e^{-\frac{\beta |\xi|^2}{4}}\ud \xi\\
    =&\int_{\R^3}{\left|x\right|^{\alpha}}e^{-\frac{|x|^2}{\beta}}\ud x\le C(\alpha,\beta)\,,
    \end{align*}
    which completes the proof.
\end{proof}

\section*{Acknowledgements}
The authors would like to thank Mickey Salins for helpful discussions. J. H. would like to thank Carl Mueller for pointing out references \cite{peszat.zabczyk:00:nonlinear, peszat:02:cauchy}.

\end{document}